%% file: Main.tex
\documentclass{article}
\usepackage{preamble} % Loads preamble.sty from the same directory
\title{Equivariant Riemann--Roch for Surfaces via Connections and Residues}
\author{Alexandros Kafkas}
\date{}
\hypersetup{
  pdfsubject={Equivariant Euler characteristics of line bundles on smooth projective surfaces}
}

\begin{document}

\maketitle

\begin{abstract}
Let $X$ be a smooth projective surface over $\mathbb C$ with an effective
action of a finite group $G$, and let $L$ be a $G$-linearized line bundle.
We develop a residue-theoretic and birational approach to computing
the equivariant Euler characteristic $\chi_G(X,L)$. Every such $L$
admits an equivariant divisor presentation. Divisor peeling applies the
divisor exact sequences one layer at a time, treating whole $G$-orbits
of components together. It reduces the problem to the structure sheaf
and lower-dimensional equivariant contributions.

On a resolution of $X/G$, the natural extensions of the isotypic
direct-image sheaves preserve Euler characteristics, but their
induced connections need not be logarithmic, as an explicit $D_4$
calculation shows. Blowing up fixed points before taking the quotient
reduces the local calculation to diagonal actions and cyclic
Hirzebruch--Jung singularities. We obtain explicit correction class
functions for the ADE singularities and a uniform procedure for
general diagonal actions, including quasi-reflections. In the
diagonal case, the natural extensions are logarithmic, and their
residues are computed from monomial valuations on the
Hirzebruch--Jung resolution. Combining these residue calculations
with the branch-curve terms and divisor peeling gives a procedure for
computing the equivariant Euler characteristic of every $G$-linearized
line bundle. We compare the local classes with the holomorphic
Lefschetz formula and the Kleinian coefficients of Lim and Rota, and
give examples on $\mathbb P^2$, including one with fixed curves.
\end{abstract}

%%% --- Table of Contents --- 

\setcounter{tocdepth}{1}
\tableofcontents

%%% --- Introduction --- 

\input{ch-introduction}

%%% --- Preliminaries --- 

\input{ch-preliminaries}

%%% --- Reduction --- 

\input{ch-reduction}

%%% --- Obstruction --- 

\input{ch-obstruction}

%%% --- blowup --- 

\input{ch-blowup}

%%% --- general hj --- 

\input{ch-general-hj}

%%% ---applications --- 

\input{ch-applications}

%%% --- Acknowledgments ---
\section*{Acknowledgments}
This paper is based on the author's doctoral dissertation at Purdue
University. The author thanks his advisor, Donu Arapura, for his help and
guidance throughout this work, and Kenji Matsuki for the many helpful conversations.

%%% --- Bibliography ---
\begingroup
\setlength{\emergencystretch}{2em}
\printbibliography[heading=bibintoc]
\endgroup

\end{document}

%% file: ch-introduction.tex
\ProvidesFile{ch-introduction.tex}

% The schematic below requires:
% \usepackage{tikz}
% \usetikzlibrary{arrows.meta,positioning}

\section{Introduction}

Let $X$ be a smooth projective surface over $\mathbb C$ with an effective action of a finite group $G$, and let $\mathcal L$ be a $G$-linearized line bundle on $X$. The linearization induces representations of $G$ on the cohomology groups $H^i(X,\mathcal L)$. The central object in equivariant Riemann--Roch is the virtual representation
\[
    \chi_G(X,\mathcal L)
    :=\sum_{i=0}^2(-1)^i[H^i(X,\mathcal L)]
    \in \operatorname{Rep}(G),
\]
where $\operatorname{Rep}(G)$ is the representation ring of finite-dimensional complex representations of $G$.

Equivariant Riemann--Roch relates equivariant $K$-theory to equivariant Chern characters and Todd classes, see Edidin--Graham \cite{edidin_riemann-roch_1999}. Their theorem uses completion at the augmentation ideal, which does not by itself retain the full character for a finite group. The full character is given by the holomorphic Lefschetz formula \cite{atiyah_lefschetz_1968,atiyah_index_1968}, and in orbifold form by Kawasaki's Riemann--Roch theorem \cite{kawasaki_riemann-roch_1979}. See also \cite[Section~2]{lim_riemannroch_2023}. These formulas compute the character from fixed-point data on $X$. Our goal is different: we express the multiplicities through connections and residues on a resolution of the quotient. For the structure sheaf, each multiplicity is then given by residues and intersection numbers on a smooth surface, without summing a character over the group. Fixed curves enter through the self-intersections and canonical degrees of their images.

For smooth projective curves, Arapura \cite{arapura_residues_2022} gives a more explicit approach. A $G$-invariant logarithmic connection on $\mathcal L$ induces logarithmic connections on the isotypic summands of the direct image over the quotient curve. Their residues determine their degrees, and Riemann--Roch on the quotient then gives the multiplicities in $\chi_G(X,\mathcal L)$.

The goal of this paper is to develop a residue-theoretic and birational version of this method for smooth projective surfaces. We first choose an equivariant divisor presentation of $\mathcal L$ and reduce the computation to the structure sheaf and lower-dimensional terms. This divisor-peeling step uses the divisor exact sequences to add or remove one layer of the divisor at a time, treating all components in the same $G$-orbit together to preserve equivariance. Repeating these exact sequences until the divisor is removed leaves $\chi_G(X,\mathcal O_X)$ and equivariant contributions from the curves. We then study the structure-sheaf contribution in two ways. The first takes the quotient and passes to a resolution. The second blows up the surface before taking the quotient. The latter replaces a non-abelian local action by diagonal local actions and leads to cyclic Hirzebruch--Jung calculations.

Explicit correction terms have also been studied from the orbifold viewpoint: Buckley--Reid--Zhou \cite{buckley_ice_2013} give formulas for Hilbert series, and Lim--Rota \cite{lim_riemannroch_2023} compute the Riemann--Roch coefficients for ADE orbisurfaces. These papers use isotropy contributions on an orbifold or stack. We instead work with connections and residues on smooth surfaces obtained by blow-ups and quotient resolutions. The known ADE formulas therefore provide a comparison between the two approaches, while the general diagonal calculation gives the local terms needed beyond these cases.

Our contribution is a systematic procedure combining divisor peeling with the continued-fraction and residue calculations on Hirzebruch--Jung resolutions. The procedure applies beyond the ADE cases, including local actions with quasi-reflections. Together with the curve contributions from divisor peeling, it computes the equivariant Euler characteristic of every $G$-linearized line bundle. The combined statement is Theorem~\ref{thm:intro-global}.

For a point $p\in X$, we write
\[
    G_p:=\{g\in G:g\cdot p=p\}
\]
for its isotropy subgroup. In the local calculations, characters label the action on coordinate functions, as in Section~\ref{sec:general-hj}. These are the inverses of the corresponding tangent characters.

\subsection{The equivariant Riemann--Roch framework}

We organize the virtual character as a regular topological term together with correction class functions. More precisely, in
$\operatorname{Rep}(G)\otimes_{\mathbb Z}\mathbb Q$ we consider a decomposition of the form
\begin{equation}
\label{eq:master-err-intro}
\begin{aligned}
    \chi_G(X,\mathcal L)
    ={}&\left(
        \frac{1}{2|G|}
        c_1(\mathcal L)\cdot
        \bigl(c_1(\mathcal L)-K_X\bigr)
        +\chi(Y,\mathcal O_Y)
      \right)\chi_{\mathrm{reg}}\\
    &-\sum_{\xi\in\operatorname{Irr}(G)}
      m_\xi(\mathcal L)\,\xi,
\end{aligned}
\end{equation}
where $Y=X/G$ and $\chi_{\mathrm{reg}}$ is the character of the regular representation. For $\mathcal L=\mathcal O_X$, the rational coefficients $m_\xi(\mathcal O_X)$ collect the contributions from the ramified locus, including fixed curves and isolated fixed points. They consist of local contributions of the exceptional configurations, together with global terms from the self-intersections and canonical degrees of the branch curves. For general $\mathcal L$, the coefficients $m_\xi(\mathcal L)$ also absorb the divisor-peeling terms. Although the two parts on the right are written with rational coefficients, their sum is the integral virtual representation $\chi_G(X,\mathcal L)$.

If $G$ acts freely, then $X\to Y$ is finite \'etale, every $G$-linearized line bundle descends to $Y$, and the correction terms vanish. Evaluation at the identity then recovers the ordinary Hirzebruch--Riemann--Roch formula on $X$, see \cite[Chapter~V, Theorem~1.6]{hartshorne_algebraic_1977} for the surface formula.

\subsection{From the curve method to the surface methods}

The constructions in this paper follow the same basic idea as Arapura's argument, but surfaces require two additional steps.

\paragraph{The quotient-resolution method.}
For a surface, one can begin in the same way with
\[
    \pi\colon X\longrightarrow Y=X/G.
\]
The quotient may be singular, and the isotypic summands of $\pi_*\mathcal O_X$ are generally reflexive rather than locally free sheaves. On a resolution, their torsion-free pullbacks are locally free and have the same Euler characteristics as the original sheaves, by Proposition~\ref{extension_prop}. The next step would be to compute their Chern classes from residues, using Ohtsuki's residue formula \cite{ohtsuki_residue_1982}.

This works when the induced connections on the natural extensions are logarithmic. For an $A_n$ singularity, put $N=n+1$ and choose a generator $g$ of the cyclic isotropy subgroup so that the coordinate functions have characters $\chi,\chi^{-1}$, with $\chi(g)=\exp(2\pi i/N)$. The residue calculation gives
\[
    m_{p,\chi^j}=\frac{j(N-j)}{2N},
    \qquad 0\leq j<N,
\]
as proved in Proposition~\ref{prop:An-local-correction}.

The natural extension is not logarithmic in general. The induced flat connection is still regular singular, so some logarithmic extension exists by Deligne's extension theorem \cite[Chapter~II]{deligne_equations_1970}, but it may differ from the natural extension. One must then compare their Euler characteristics, including a term supported on the exceptional divisor. The $D_4$ calculation gives a concrete example of this obstruction.

\paragraph{The pre-quotient blow-up method.}
The comparison above becomes difficult for larger non-abelian isotropy subgroups. We therefore use a second method. We first blow up a fixed point $p\in X$ and only then take the quotient. The exceptional curve is
\[
    E=\mathbb P(T_{X,p})\cong\mathbb P^1.
\]
The isotropy subgroups of points on $E$ act diagonally in coordinates adapted to $E$. After quotienting by their quasi-reflection subgroups, the residual groups are cyclic. The original non-abelian local problem is therefore replaced by cyclic Hirzebruch--Jung calculations attached to the special isotropy orbits on $E$.

For a non-cyclic binary polyhedral group, there are three special isotropy orbits. Each gives a cyclic arm meeting the central curve. The resulting local corrections are transported back to $G_p$ by induction, with the central-curve term included only once. This gives the ADE assembly formula of Theorem~\ref{thm:blowup-assembly}. In these noncyclic \(D\)- and \(E\)-cases, the pre-quotient blow-up followed by resolution gives the same minimal resolution as the quotient-resolution method. However, it gives a different extension of the multiplicity sheaves, and this extension is logarithmic.

The same method applies when the diagonal isotropy subgroup contains quasi-reflections in one or both coordinate directions. The continued fraction determines the intersection matrix and the canonical intersections. The character congruences give a finite list of monomial generators, whose minimum valuations determine the residue vectors. Thus the diagonal calculation gives an explicit finite procedure without requiring a separate logarithmic modification. Section~\ref{sec:general-hj} specifies how the exceptional contributions and the remaining fixed-curve terms enter the global formula.

Figure~\ref{fig:method-overview} summarizes the three constructions. The dashed arrow in the middle row denotes a comparison of extensions, not a canonical identification.

\begin{figure}[ht]
\centering
\begin{tikzpicture}[
    font=\small,
    >=Latex,
    node distance=5mm and 4mm,
    methodlabel/.style={
        align=right,
        font=\small\bfseries,
        text width=20mm
    },
    methodbox/.style={
        draw,
        rounded corners=2pt,
        align=center,
        text width=22mm,
        minimum height=12mm,
        inner sep=2.5pt
    },
    arrow/.style={->,thick},
    compare/.style={->,thick,dashed}
]

% Curve method
\node[methodlabel] (curve-label) {Curve\\method};
\node[methodbox, below=of curve-label] (curve-quotient)
    {$Z=C/G$\\smooth};
\node[methodbox, right=of curve-quotient] (curve-isotypic)
    {isotypic\\vector bundles};
\node[methodbox, right=of curve-isotypic] (curve-residues)
    {logarithmic\\residues};
\node[methodbox, right=of curve-residues] (curve-output)
    {degrees and\\Riemann--Roch};

\draw[arrow] (curve-quotient) -- (curve-isotypic);
\draw[arrow] (curve-isotypic) -- (curve-residues);
\draw[arrow] (curve-residues) -- (curve-output);

% Quotient-resolution method
\node[methodlabel, below=22mm of curve-label] (resolution-label)
    {Quotient--\\resolution};
\node[methodbox, below=of resolution-label] (surface-quotient)
    {$Y=X/G$\\possibly singular};
\node[methodbox, right=of surface-quotient] (surface-reflexive)
    {reflexive\\isotypic sheaves};
\node[methodbox, right=of surface-reflexive] (surface-natural)
    {natural extension\\on $\widetilde Y$};
\node[methodbox, right=of surface-natural] (surface-log)
    {logarithmic\\extension};

\draw[arrow] (surface-quotient) -- (surface-reflexive);
\draw[arrow] (surface-reflexive) -- (surface-natural);
\draw[compare] (surface-natural) -- node[above=7mm,font=\scriptsize] {compare} (surface-log);

\node[methodbox, below=5mm of surface-log] (surface-output)
    {residues plus\\exceptional correction};
\draw[arrow] (surface-log) -- (surface-output);

% Pre-quotient method
\node[methodlabel, below=24mm of resolution-label] (blowup-label)
    {Pre-quotient\\blow-up};
\node[methodbox, below=of blowup-label] (blowup-x)
    {blow up\\$p\in X$};
\node[methodbox, right=of blowup-x] (blowup-e)
    {$E=\mathbb P(T_{X,p})$\\isotropy orbits};
\node[methodbox, right=of blowup-e] (blowup-cyclic)
    {cyclic local\\quotients};
\node[methodbox, right=of blowup-cyclic] (blowup-output)
    {assemble local\\corrections};

\draw[arrow] (blowup-x) -- (blowup-e);
\draw[arrow] (blowup-e) -- (blowup-cyclic);
\draw[arrow] (blowup-cyclic) -- (blowup-output);

\end{tikzpicture}
\caption{Arapura's curve method and the two surface constructions.}
\label{fig:method-overview}
\end{figure}
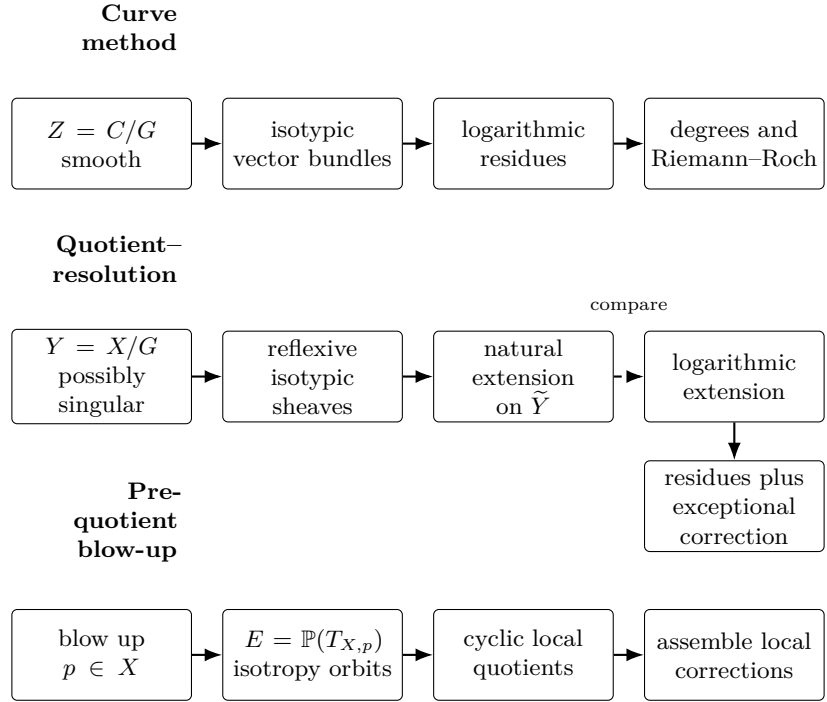

\subsection{The global formula}

The divisor-peeling argument applies after fixing an equivariant isomorphism $\mathcal L\simeq\mathcal O_X(D)$, where $D$ is a $G$-invariant Cartier divisor and $\mathcal O_X(D)$ carries its natural divisor linearization. Since the action is effective, it is generically free. Hilbert's theorem 90, in the form of Galois descent for the extension $\mathbb C(X)/\mathbb C(X)^G$, therefore gives a nonzero $G$-invariant rational section of every $G$-linearized line bundle \cite[Tag~0CDR]{authors_stacks_2018}. Its divisor gives the required equivariant presentation, so choosing $D$ imposes no additional restriction on $\mathcal L$. Section~\ref{sec:reduction} reduces the calculation to the structure sheaf, smooth curves, and the internal nodes of the orbit divisors.

The preceding constructions can be combined into one statement. We use the full logarithmic divisor on the resolved quotient, so fixed-curve contributions are included along with the exceptional terms.

\begin{theorem}[Global equivariant Euler characteristic]
\label{thm:intro-global}
Let $X$ be a smooth projective surface over $\mathbb C$ with an effective action of a finite group $G$, and let $\mathcal L$ be a $G$-linearized line bundle. Fix an equivariant divisor presentation
\[
    \mathcal L\simeq\mathcal O_X(D),
\]
with the natural divisor linearization. There is a sequence of blow-ups of finite $G$-orbits
\[
    f\colon X'\longrightarrow X
\]
such that $f^*D$ has SNC support, the isotropy subgroup at every point of $X'$ acts diagonally in local coordinates, and distinct fixed curves in the same $G$-orbit are disjoint. Write
\[
    f^*D=\sum_{j=1}^r m_jB_j,
    \qquad P_j=\sum_{\ell<j}m_\ell B_\ell,
\]
where the $B_j$ are the ordered reduced orbit divisors. Let
\[
    \pi'\colon X'\longrightarrow Y':=X'/G,
    \qquad \mu\colon S\longrightarrow Y'
\]
be the quotient map and its minimal resolution. Let $\Delta=\sum_a\Delta_a$ be the reduced divisor consisting of the exceptional curves of $\mu$ and the strict transforms of the divisorial branch locus.

For $\xi\in\operatorname{Irr}(G)$, choose a representation $W_\xi$ with character $\xi$ and set
\[
    \mathcal F_\xi
      :=\operatorname{Hom}_G(W_\xi,\pi'_*\mathcal O_{X'}),
    \qquad
    \mathcal V_\xi
      :=\mu^*\mathcal F_\xi/\operatorname{torsion}.
\]
Then $\Delta$ is SNC, and $\mathcal V_\xi$ is a rank-$\xi(1)$ vector bundle whose induced connection is logarithmic along $\Delta$. Write $R_{\xi,a}$ for its residue along $\Delta_a$, and put
\[
\begin{aligned}
    m_\xi(\mathcal O_X)
    :={}&-\frac12\sum_a\Bigl(
       \operatorname{tr}(R_{\xi,a}^2)\Delta_a^2
       +\operatorname{tr}(R_{\xi,a})K_S\cdot\Delta_a
       \Bigr)\\
    &-\sum_{a<b}\sum_{z\in\Delta_a\cap\Delta_b}
       \operatorname{tr}\bigl(R_{\xi,a}(z)R_{\xi,b}(z)\bigr).
\end{aligned}
\]
Then, with $Y=X/G$,
\begin{equation}
\label{eq:intro-global-formula}
\begin{aligned}
    \chi_G(X,\mathcal L)
    ={}&\chi(Y,\mathcal O_Y)\chi_{\mathrm{reg}}
       -\sum_{\xi\in\operatorname{Irr}(G)}
           m_\xi(\mathcal O_X)\,\xi\\
    &+\sum_{\substack{1\leq j\leq r\\m_j>0}}
       \sum_{k=1}^{m_j}
       \chi_G\!\left(B_j,\mathcal O_{B_j}(P_j+kB_j)\right)\\
    &-\sum_{\substack{1\leq j\leq r\\m_j<0}}
       \sum_{k=0}^{-m_j-1}
       \chi_G\!\left(B_j,\mathcal O_{B_j}(P_j-kB_j)\right).
\end{aligned}
\end{equation}
The residues are computed from the diagonal character modules by the finite monomial-valuation formula of Section~\ref{sec:general-hj}. The orbit-divisor terms are computed from their smooth components and internal nodes by \eqref{eq:surface-normalization}.
\end{theorem}

The proof is given at the end of Section~\ref{sec:general-hj}. Thus the calculation uses global intersection numbers on $S$, finite local character and continued-fraction calculations, and the lower-dimensional contributions from divisor peeling. The examples in Section~\ref{sec:applications} carry out the structure-sheaf calculation and divisor peeling on $\mathbb P^2$.

The coefficients $m_\xi(\mathcal L)$ in
\eqref{eq:master-err-intro} are also determined by the global formula.
Comparing \eqref{eq:master-err-intro} with
\eqref{eq:intro-global-formula}, we obtain $m_\xi(\mathcal L)$ by
starting with $m_\xi(\mathcal O_X)$, adding
\[
    \frac{\xi(1)}{2|G|}
    c_1(\mathcal L)\cdot
    \bigl(c_1(\mathcal L)-K_X\bigr),
\]
and subtracting the multiplicity of $\xi$ in the total signed
divisor-peeling contribution. Here the two peeling sums are taken
with their signs in \eqref{eq:intro-global-formula}, and the factor
$\xi(1)$ is the multiplicity of $\xi$ in the regular representation.

\subsection{Organization of the paper}

Section~\ref{sec:preliminaries} goes over the representation-theoretic, connection-theoretic, and surface-theoretic facts used in the paper. It includes the natural-extension result of Proposition~\ref{extension_prop} and the residue theorem used on resolutions.

In Section~\ref{sec:reduction}, we prove the ordered divisor-peeling formula. We first establish equivariant birational invariance, then treat positive and negative coefficients, and finally analyze the internal intersections of an orbit divisor by equivariant SNC inclusion--exclusion.

In Section~\ref{sec:obstruction}, we apply the quotient-resolution method to cyclic $A_n$ singularities and then exhibit the failure of logarithmicity for the natural extension associated with the $D_4$ singularity.

In Section~\ref{sec:blowup}, we develop the pre-quotient blow-up method and compute the correction classes for the binary polyhedral isotropy subgroups. Section~\ref{sec:general-hj} treats the general diagonal calculation with quasi-reflections, shows that the local classes do not depend on the model, and ends with the proof of Theorem~\ref{thm:intro-global}.

Finally, in Section~\ref{sec:applications}, we compare the local classes with the holomorphic Lefschetz formula, recover the cyclic and $D_4$ correction coefficients of Lim--Rota, and work out a global example on $\mathbb P^2$, using the local calculation for $\frac15(1,2)$ and divisor reduction for $\mathcal O_{\mathbb P^2}(d)$ to the structure sheaf. We end with an $S_3$-action on $\mathbb P^2$ in which the pre-quotient blow-up and the branch-curve terms both enter.

%% file: ch-preliminaries.tex
\ProvidesFile{ch-preliminaries.tex}

\section{Preliminaries}
\label{sec:preliminaries}

This section fixes the notation and recalls the results used later. We include only the material needed for equivariant divisor peeling, quotient resolutions, logarithmic residue calculations, and pre-quotient blow-ups.

Throughout, varieties are defined over $\mathbb C$, and $G$ is a finite group. Unless stated otherwise, varieties are irreducible and the action of $G$ is effective. All representations are finite-dimensional complex representations. For a point $p$ of a $G$-variety, we write
\[
    G_p:=\{g\in G:g\cdot p=p\}
\]
for its isotropy subgroup.

\subsection{Representations and isotypic components}

We recall the representation theory we use following \cite{serre_linear_1977}. Let $\operatorname{Rep}(G)$ be the Grothendieck ring of finite-dimensional complex representations of $G$, and let $\operatorname{Irr}(G)$ denote the set of irreducible characters. For class functions $\alpha$ and $\beta$, we use the inner product
\[
    \langle \alpha,\beta\rangle_G
    :=\frac{1}{|G|}\sum_{g\in G}
      \alpha(g)\overline{\beta(g)}.
\]
If $V$ is a $G$-representation and $\xi\in\operatorname{Irr}(G)$, then
\[
    \langle\chi_V,\xi\rangle_G
\]
is the multiplicity of the irreducible representation with character $\xi$ in $V$.

Choose an irreducible representation $W_\xi$ with character $\xi$. The $\xi$-isotypic component of $V$ is the sum of all irreducible subrepresentations isomorphic to $W_\xi$. It has the canonical description
\[
    V_\xi
    \cong
    W_\xi\otimes\operatorname{Hom}_G(W_\xi,V).
\]
Here $\operatorname{Hom}_G(W_\xi,V)$ is the multiplicity space, whose dimension is $\langle\chi_V,\xi\rangle_G$. The component $V_\xi$ is the image of the central idempotent
\[
    e_\xi
    =\frac{\xi(1)}{|G|}
      \sum_{g\in G}\xi(g^{-1})g
    \in\mathbb C[G].
\]
Thus
\[
    V=\bigoplus_{\xi\in\operatorname{Irr}(G)}V_\xi
\]
canonically. A further decomposition of $V_\xi$ into individual irreducible summands is generally not canonical.

For a subgroup $H\leq G$, we write $\operatorname{Res}_H^G$ and $\operatorname{Ind}_H^G$ for restriction and induction. Frobenius reciprocity gives
\[
    \left\langle
      \operatorname{Ind}_H^G\alpha,\beta
    \right\rangle_G
    =
    \left\langle
      \alpha,\operatorname{Res}_H^G\beta
    \right\rangle_H.
\]
This identity is used repeatedly to transport local class functions from an isotropy subgroup to the ambient group.

\subsection{Equivariant bundles and finite quotients}

Let $X$ be a smooth projective variety with a left action of $G$. A $G$-linearization of a vector bundle $V$ consists of isomorphisms
\[
    \phi_g\colon g^*V\xrightarrow{\sim}V,
    \qquad g\in G,
\]
satisfying
\[
    \phi_{gh}=\phi_h\circ h^*\phi_g.
\]
Here $(gh)^*=h^*g^*$. Equivalently, the action of $G$ on $X$ lifts to the total space of $V$ and is linear on each fiber. A bundle with such a linearization is called $G$-linearized or $G$-equivariant.

The linearization induces $G$-actions on the cohomology groups $H^i(X,V)$. We define
\[
    \chi_G(X,V)
    :=\sum_i(-1)^i[H^i(X,V)]
    \in\operatorname{Rep}(G).
\]

Let
\[
    \pi\colon X\longrightarrow Y:=X/G
\]
be the quotient map. Since $G$ is finite, the quotient exists as a projective variety, $\pi$ is finite and surjective, and $Y$ is normal. The construction of the quotient variety is done by choosing an affine open cover and taking the invariant rings. See \cite{mumford_geometric_1994,dolgachev_lectures_2003} for more details on these facts.

If $\mathcal F$ is a $G$-equivariant coherent sheaf on $X$, then $\pi_*\mathcal F$ is a coherent sheaf on $Y$ with an $\mathcal O_Y$-linear action of $G$. The central idempotents give a canonical decomposition
\begin{equation}
\label{eq:prelim-isotypic-sheaves}
    \pi_*\mathcal F
    =\bigoplus_{\xi\in\operatorname{Irr}(G)}
      (\pi_*\mathcal F)_\xi,
    \qquad
    (\pi_*\mathcal F)_\xi
    :=e_\xi(\pi_*\mathcal F).
\end{equation}
Because $\pi$ is finite,
\[
    H^i(X,\mathcal F)
    \cong H^i(Y,\pi_*\mathcal F)
\]
as $G$-representations. Therefore
\begin{equation}
\label{eq:prelim-isotypic-euler}
    \chi\bigl(Y,(\pi_*\mathcal F)_\xi\bigr)
    =\xi(1)
      \left\langle
        \chi_G(X,\mathcal F),\xi
      \right\rangle_G.
\end{equation}

An effective action of a finite group is free over a dense open subset of $X$. Hence, if $L$ is a $G$-linearized line bundle, the generic fiber of $\pi_*L$ is the regular representation. In particular,
\[
    \operatorname{rk}(\pi_*L)_\xi=\xi(1)^2.
\]
If $\mathcal F$ is locally free and $Y$ is a smooth curve, the sheaves $(\pi_*\mathcal F)_\xi$ are vector bundles. When $X$ is a smooth surface, the isotypic summands of $\pi_*\mathcal F$ are reflexive sheaves. They need not be locally free at the quotient singularities.

For the structure sheaf, we also use the multiplicity sheaf
\[
    \mathcal F_\xi
      :=\operatorname{Hom}_G(W_\xi,\pi_*\mathcal O_X).
\]
The two descriptions are related by
\[
    (\pi_*\mathcal O_X)_\xi
      \cong W_\xi\otimes_{\mathbb C}\mathcal F_\xi,
    \qquad \operatorname{rk}\mathcal F_\xi=\xi(1).
\]
Here the rank is $\xi(1)$, while the full isotypic summand has rank $\xi(1)^2$. The Euler characteristic of $\mathcal F_\xi$ gives the multiplicity directly:
\[
    \chi(Y,\mathcal F_\xi)
      =\left\langle\chi_G(X,\mathcal O_X),\xi\right\rangle_G.
\]
This is the sheaf used in Theorem~\ref{thm:intro-global}.

\subsubsection{The natural linearization on a divisor line bundle}

If $D$ is a $G$-invariant Cartier divisor, then $\mathcal O_X(D)$ has a natural $G$-linearization induced by the action on rational functions. We call this the \emph{natural divisor linearization}. This term refers to the linearization of $\mathcal O_X(D)$. It does not refer to the canonical divisor $K_X$ or the canonical sheaf $\omega_X$.

The natural divisor linearization fixes the distinguished rational section $s_D$ whose divisor is $D$. An arbitrary linearization of the underlying line bundle need not be the natural divisor linearization. For example, if $G$ acts trivially on $X$, then the trivial line bundle may be linearized by any character of $G$.

\subsection{Connections, logarithmic poles, and residues}

We use the standard definitions of logarithmic connections and residues
from \cite{deligne_equations_1970}. Let $V$ be a vector bundle on a smooth variety $X$. A connection on $V$ is a $\mathbb C$-linear map
\[
    \nabla\colon V\longrightarrow V\otimes\Omega_X^1
\]
satisfying
\[
    \nabla(fs)=f\nabla(s)+s\otimes df.
\]
In a local frame, one may write
\[
    \nabla=d+A,
\]
where $A$ is a matrix of one-forms. Its curvature is
\[
    \Theta_\nabla=dA+A\wedge A.
\]
The connection is \emph{flat}, or \emph{integrable}, if $\Theta_\nabla=0$.

Suppose that $V$ is $G$-linearized by isomorphisms $\phi_g\colon g^*V\to V$. Let
\[
    \kappa_g\colon g^*\Omega_X^1\xrightarrow{\sim}\Omega_X^1
\]
be the isomorphism induced by the differential of $g$. The connection is $G$-equivariant if
\[
    \nabla\circ\phi_g
    =
    (\phi_g\otimes\kappa_g)\circ g^*\nabla
    \qquad(g\in G).
\]
The structure sheaf with its natural linearization carries the $G$-equivariant flat connection
\[
    d\colon\mathcal O_X\longrightarrow\Omega_X^1.
\]
This is the connection used after the reduction to $\mathcal O_X$.

\subsubsection{Logarithmic connections and residues}

Let
\[
    D=\sum_{i=1}^mD_i
\]
be a reduced simple normal crossings divisor. Locally, one can choose coordinates $z_1,\ldots,z_n$ such that
\[
    D=\{z_1\cdots z_k=0\}.
\]
Then $\Omega_X^1(\log D)$ is locally free with basis
\[
    \frac{dz_1}{z_1},\ldots,
    \frac{dz_k}{z_k},
    dz_{k+1},\ldots,dz_n.
\]

A logarithmic connection on $V$ with poles along $D$ is a connection
\[
    \nabla\colon V\longrightarrow
    V\otimes\Omega_X^1(\log D).
\]
Near the generic point of $D_i$, its matrix has the form
\[
    A=R_i\frac{dz_i}{z_i}+B,
\]
where $B$ is regular in the direction normal to $D_i$. The residue along $D_i$ is
\[
    \operatorname{Res}_{D_i}(\nabla)
    :=R_i|_{D_i}
    \in
    H^0\bigl(D_i,\operatorname{End}(V|_{D_i})\bigr).
\]
Under a change of frame, the residue is conjugated. Its characteristic polynomial and the traces of its powers are therefore intrinsic. If $\nabla$ is flat, then the residues along two components commute after restriction to their intersection.

\subsection{Residues and Chern classes on a surface}

We recall only the form of the residue theorem used later. Let $S$ be a smooth projective surface, let
\[
    D=\sum_iD_i
\]
be an SNC divisor, and let $V$ be a rank-$r$ vector bundle carrying a flat logarithmic connection with residue $R_i$ along $D_i$. Then, in complex cohomology,
\[
    c_1(V)
    =-\sum_i\operatorname{tr}(R_i)[D_i]
    \in H^2(S,\mathbb C),
\]
and
\[
    \operatorname{ch}_2(V)
    =\frac12\sum_i\operatorname{tr}(R_i^2)[D_i]^2
      +\sum_{i<j}\sum_{p\in D_i\cap D_j}
        \operatorname{tr}\bigl(R_i(p)R_j(p)\bigr)[p]
    \in H^4(S,\mathbb C).
\]
For $i\neq j$, the product $R_i(p)R_j(p)$ is evaluated at each point of $D_i\cap D_j$. Flatness implies that the restricted residues commute. For $i=j$, the term is $\operatorname{tr}(R_i^2)D_i^2$. These identities are the surface case of the residue formula of Ohtsuki and Esnault–Viehweg \cite{ohtsuki_residue_1982,esnault_logarithmic_1986}.

Hirzebruch--Riemann--Roch \cite{fulton_intersection_1984} gives
\[
    \chi(S,V)
    =r\chi(\mathcal O_S)
     -\frac12c_1(V)\cdot K_S
     +\operatorname{ch}_2(V).
\]
Substituting the residue formulas yields
\begin{equation}
\label{eq:prelim-residue-euler}
\begin{aligned}
    \chi(S,V)
    ={}&r\chi(\mathcal O_S)
      +\frac12\sum_i
        \operatorname{tr}(R_i)D_i\cdot K_S\\
     &+\frac12\sum_i\operatorname{tr}(R_i^2)D_i^2
       +\sum_{i<j}\sum_{p\in D_i\cap D_j}
          \operatorname{tr}\bigl(R_i(p)R_j(p)\bigr).
\end{aligned}
\end{equation}

On a smooth projective curve $C$, the corresponding residue theorem is
\[
    \deg(V)
    =-\sum_{p\in C}
      \operatorname{tr}\operatorname{Res}_p(\nabla).
\]
This is the formula used in Arapura's argument
\cite[Theorem~3.2]{arapura_residues_2022}.

\subsection{Regular singular direct-image connections}

Let $\pi\colon U\to V$ be a finite \'etale morphism of smooth varieties, and let $(L,\nabla)$ be a flat bundle on $U$. Then $\pi_*L$ carries the induced flat connection. In the applications below, $(L,\nabla)$ is the trivial connection $(\mathcal O_U,d)$, which is regular singular. Since $\pi$ is finite, and thus proper, its direct-image connection is again regular singular \cite{katz_regularity_1970}.

A flat algebraic connection is regular singular if, on one, and hence
every, smooth compactification with SNC boundary, it admits a locally
free logarithmic extension \cite{deligne_equations_1970}. Regular singularity guarantees the existence of a logarithmic extension. It does not identify that extension with a separately defined natural extension.

For a finite quotient $\pi\colon X\to Y=X/G$, let $B\subset Y$ be the reduced divisor branch locus and put
\[
    Y^\circ:=Y_{\mathrm{reg}}\setminus B.
\]
Over $Y^\circ$, the quotient is finite \'etale, and the isotypic direct-image bundles carry flat connections induced by
\[
    d\colon\mathcal O_X\longrightarrow\Omega_X^1.
\]
On a resolution, the relevant logarithmic boundary consists of the exceptional divisor together with the strict transform of $B$.

\subsection{Reflexive sheaves and the natural extension}

Let $Y$ be a normal projective surface with rational singularities, and let
\[
    \mu\colon\widetilde Y\longrightarrow Y
\]
be a resolution that is an isomorphism over $Y_{\mathrm{reg}}$. For a reflexive sheaf $\mathcal E$ on $Y$, its restriction to $Y_{\mathrm{reg}}$ is locally free and
\[
    \mathcal E
    \cong j_*\bigl(\mathcal E|_{Y_{\mathrm{reg}}}\bigr),
    \qquad
    j\colon Y_{\mathrm{reg}}\hookrightarrow Y.
\]
Thus $\mathcal E$ is determined by its restriction to the smooth locus,
see \cite[Tag~0EBJ]{authors_stacks_2018}.

The extension used below is the torsion-free pullback
\[
    \mathcal V^{\mathrm{nat}}
    :=\mu^*\mathcal E/\operatorname{torsion}.
\]
For rational surface singularities, this sheaf is locally free and preserves the Euler characteristic.

\begin{proposition}[Natural extension over a surface with rational singularities]
\label{extension_prop}
Let $Y$ be a normal projective surface over $\mathbb C$ with rational
singularities, let
\[
    \mu\colon\widetilde Y\longrightarrow Y
\]
be a resolution that is an isomorphism over $Y_{\mathrm{reg}}$, and let
$\mathcal E$ be a coherent reflexive sheaf of rank $r$ on $Y$. Set
\[
    \mathcal V^{\mathrm{nat}}
    :=\mu^*\mathcal E/\operatorname{torsion}.
\]
Then:
\begin{enumerate}
    \item $\mathcal V^{\mathrm{nat}}$ is locally free of rank $r$,
    \item $\mu_*\mathcal V^{\mathrm{nat}}\cong\mathcal E$, and
    \item $R^i\mu_*\mathcal V^{\mathrm{nat}}=0$ for every $i>0$.
\end{enumerate}
Consequently,
\[
    R\mu_*\mathcal V^{\mathrm{nat}}\simeq\mathcal E
\]
and
\begin{equation}
\label{eq:prelim-natural-euler}
    \chi(\widetilde Y,\mathcal V^{\mathrm{nat}})
    =\chi(Y,\mathcal E).
\end{equation}
\end{proposition}

\begin{proof}[Outline of proof]
The assertions are local on $Y$ and immediate over $Y_{\mathrm{reg}}$,
so we may work over an affine neighborhood of a singular point.
Local freeness holds for every rational surface singularity by
Artin--Verdier \cite[Lemma~1.1]{artin_reflexive_1985}. For quotient
singularities, see also Esnault \cite{esnault_reflexive_1985} and Wunram
\cite{wunram_reflexive_1988}, and for partial resolutions, Gustavsen--Ile
\cite[Proposition~2.7]{gustavsen_deformations_2018}. The rank is $r$ because $\mu$ is an
isomorphism over $Y_{\mathrm{reg}}$.

Adjunction gives a map $\mathcal E\to\mu_*\mu^*\mathcal E\to
\mu_*\mathcal V^{\mathrm{nat}}$. Since $\mu_*\mathcal V^{\mathrm{nat}}$ is
torsion-free, it embeds in $j_*$ of its restriction to $Y_{\mathrm{reg}}$,
which is $j_*(\mathcal E|_{Y_{\mathrm{reg}}})\cong\mathcal E$ because
$\mathcal E$ is reflexive. The composite is the identity, so
\[
    \mu_*\mathcal V^{\mathrm{nat}}\cong\mathcal E.
\]

For the vanishing, we use rationality and generation by global
sections, as in \cite{lipman_rational_1969}.
Pulling back a finite set of generators of $\mathcal E$ gives
an exact sequence
\[
    0\longrightarrow\mathcal K
    \longrightarrow\mathcal O_{\widetilde Y}^{\oplus N}
    \longrightarrow\mathcal V^{\mathrm{nat}}
    \longrightarrow0.
\]
Rationality gives $R^1\mu_*\mathcal O_{\widetilde Y}=0$, while
$R^2\mu_*\mathcal K=0$ because the fibers have dimension at most one.
The associated long exact sequence therefore gives
$R^1\mu_*\mathcal V^{\mathrm{nat}}=0$.
The higher direct images vanish by the same fiber-dimension bound.

The derived pushforward identity follows, and Leray gives
\eqref{eq:prelim-natural-euler}.
\end{proof}

\subsection{Natural and logarithmic extensions}

Let $\pi\colon X\to Y=X/G$ be a finite quotient, where $X$ is a smooth surface. Let $B\subset Y$ be its reduced branch locus, and put
\[
    Y^\circ:=Y_{\mathrm{reg}}\setminus B.
\]
Let $\mu\colon\widetilde Y\to Y$ be the minimal resolution. Assume that
\[
    \Delta
    :=\operatorname{Exc}(\mu)+\widetilde B
\]
is an SNC divisor, where $\widetilde B$ is the strict transform of $B$.

Suppose that $\mathcal E$ is an isotypic summand of $\pi_*\mathcal O_X$. Over $Y^\circ$, it is a vector bundle carrying the regular singular flat connection induced by
\[
    d\colon\mathcal O_X\longrightarrow\Omega_X^1.
\]
Its natural extension is
\[
    \mathcal V^{\mathrm{nat}}
    =\mu^*\mathcal E/\operatorname{torsion}.
\]
By Proposition~\ref{extension_prop},
\[
    \chi(Y,\mathcal E)
    =\chi(\widetilde Y,\mathcal V^{\mathrm{nat}}).
\]
The remaining issue is the pole order of the induced connection: it need not be logarithmic on $\mathcal V^{\mathrm{nat}}$ along the exceptional divisor.

Regular singularity gives a locally free logarithmic extension
\[
    (\mathcal V^{\log},\nabla^{\log})
\]
on $\widetilde Y$, with poles along $\Delta$. In the applications below, $\mathcal V^{\log}$ is chosen to agree with the natural direct-image lattice away from the exceptional divisor
\[
    E:=\operatorname{Exc}(\mu),
\]
including along the strict transform of the divisorial branch locus. Therefore the difference
\[
    [\mathcal V^{\log}]
    -[\mathcal V^{\mathrm{nat}}]
\]
in $K_0(\widetilde Y)$ is represented by a class supported on $E$. Hence
\begin{equation}
\label{eq:prelim-extension-comparison}
    \chi(Y,\mathcal E)
    =\chi(\widetilde Y,\mathcal V^{\log})
     -\chi\bigl(
       [\mathcal V^{\log}]
       -[\mathcal V^{\mathrm{nat}}]
      \bigr).
\end{equation}
If there is an exact sequence
\[
    0\longrightarrow
    \mathcal V^{\mathrm{nat}}
    \longrightarrow
    \mathcal V^{\log}
    \longrightarrow
    \mathcal Q
    \longrightarrow0,
\]
where $\mathcal Q$ is supported on $E$, then
\[
    \chi(Y,\mathcal E)
    =\chi(\widetilde Y,\mathcal V^{\log})
     -\chi(\widetilde Y,\mathcal Q).
\]
This is the exceptional comparison term used when the natural extension is not logarithmic.

%% file: ch-reduction.tex
\ProvidesFile{ch-reduction.tex}

\section{Reduction to the Structure Sheaf by Equivariant Divisor Peeling}
\label{sec:reduction}

We now prove the divisor-reduction statement announced in the introduction. The main applications concern surfaces, but the peeling argument works in every dimension.

Let $X$ be a smooth projective variety over $\mathbb C$ with an action of a finite group $G$, not necessarily effective. We assume that the $G$-linearized line bundle $L$ has been given an equivariant divisor presentation
\[
    L\simeq \mathcal O_X(D),
\]
where $D$ is a $G$-invariant Cartier divisor and $\mathcal O_X(D)$ carries its natural divisor linearization.

Such a presentation is the same as a nonzero $G$-invariant rational section of $L$. By Hilbert's theorem 90, one exists when the action is generically free, and in particular for every effective finite action on an irreducible variety. It may fail to exist for non-effective actions. For example, if $G$ acts trivially on $X$ and the trivial line bundle is linearized by a nontrivial character, there is no nonzero invariant rational section.

\subsection{Equivariant log resolution and birational invariance}

Choose a $G$-equivariant log resolution
\[
    f\colon \widetilde X\longrightarrow X
\]
of the support of $D$. Thus $\widetilde X$ is smooth and $\operatorname{Supp}(f^*D)$ is a simple normal crossings divisor. For a surface, $f$ may be chosen as a sequence of blow-ups of finite $G$-orbits of points. In higher dimensions, smooth $G$-invariant centers may also be needed.

\begin{lemma}[Equivariant birational invariance]
\label{lem:birational-invariance}
Let $f\colon \widetilde X\to X$ be a $G$-equivariant proper
birational morphism between smooth projective varieties, and let $L$
be a $G$-linearized line bundle on $X$. Then
\[
H^i(X,L)\cong H^i(\widetilde X,f^*L)
\]
as $G$-representations for every $i$. Consequently,
\[
\chi_G(X,L)=\chi_G(\widetilde X,f^*L).
\]
\end{lemma}

\begin{proof}
Since $X$ is smooth,
\[
f_*\mathcal O_{\widetilde X}\cong\mathcal O_X,
\qquad
R^qf_*\mathcal O_{\widetilde X}=0
\quad (q>0).
\]
These isomorphisms are $G$-equivariant because $f$ is
$G$-equivariant. The projection formula therefore gives
\[
f_*(f^*L)\cong L,
\qquad
R^qf_*(f^*L)=0
\quad (q>0),
\]
equivariantly. The $G$-equivariant Leray spectral sequence degenerates,
giving the required cohomology isomorphisms. Taking alternating sums
proves the Euler characteristic identity.
\end{proof}

We may therefore replace $(X,D)$ by $(\widetilde X,f^*D)$. From now on, we assume that $D$ has simple normal crossings support.

\subsection{Ordered orbit-by-orbit peeling}

Write
\[
    D=\sum_{j=1}^r m_jB_j,
\]
where $m_j\in\mathbb Z\setminus\{0\}$ and the $B_j$ are the reduced divisors determined by the distinct $G$-orbits of irreducible components of $\operatorname{Supp}(D)$. Choose a component $C_j\subset B_j$ and let
\[
    H_j:=\{g\in G:gC_j=C_j\}
\]
be its setwise isotropy subgroup. Then
\[
    B_j=\sum_{g\in G/H_j}gC_j.
\]

Fix an ordering $B_1,\ldots,B_r$, and put
\[
    P_j:=\sum_{\ell<j}m_\ell B_\ell,
    \qquad P_1=0.
\]
Thus $P_j$ is the part of $D$ already processed before the $j$-th step. The individual terms below depend on the ordering, but their sum does not.

Suppose first that $m_j>0$. For $1\leq k\leq m_j$, the exact sequence
\[
    0\longrightarrow
    \mathcal O_X(P_j+(k-1)B_j)
    \longrightarrow
    \mathcal O_X(P_j+kB_j)
    \longrightarrow
    \mathcal O_{B_j}(P_j+kB_j)
    \longrightarrow0
\]
gives
\begin{equation}
\label{eq:positive-peeling}
\begin{aligned}
    &\chi_G\bigl(X,\mathcal O_X(P_j+m_jB_j)\bigr)
     -\chi_G\bigl(X,\mathcal O_X(P_j)\bigr)\\
    &\hspace{25mm}=
      \sum_{k=1}^{m_j}
      \chi_G\bigl(
        B_j,\mathcal O_{B_j}(P_j+kB_j)
      \bigr).
\end{aligned}
\end{equation}

Now suppose that $m_j<0$, and write $n_j:=-m_j>0$. For $0\leq k<n_j$, the exact sequence
\[
    0\longrightarrow
    \mathcal O_X(P_j-(k+1)B_j)
    \longrightarrow
    \mathcal O_X(P_j-kB_j)
    \longrightarrow
    \mathcal O_{B_j}(P_j-kB_j)
    \longrightarrow0
\]
gives
\begin{equation}
\label{eq:negative-peeling}
\begin{aligned}
    &\chi_G\bigl(X,\mathcal O_X(P_j-n_jB_j)\bigr)
     -\chi_G\bigl(X,\mathcal O_X(P_j)\bigr)\\
    &\hspace{25mm}=
      -\sum_{k=0}^{n_j-1}
      \chi_G\bigl(
        B_j,\mathcal O_{B_j}(P_j-kB_j)
      \bigr).
\end{aligned}
\end{equation}

The different index ranges are forced by the exact sequences: a positive layer ends at $P_j+kB_j$, while a negative layer begins at $P_j-kB_j$ and passes to $P_j-(k+1)B_j$.

Telescoping over the ordered list gives the basic reduction formula.

\begin{proposition}[Ordered equivariant divisor peeling]
\label{prop:ordered-peeling}
With the notation above,
\begin{align*}
    \chi_G(X,\mathcal O_X(D))
    ={}&\chi_G(X,\mathcal O_X)\\
    &+\sum_{\substack{1\leq j\leq r\\m_j>0}}
      \sum_{k=1}^{m_j}
      \chi_G\!\left(
        B_j,\mathcal O_{B_j}(P_j+kB_j)
      \right)\\
    &-\sum_{\substack{1\leq j\leq r\\m_j<0}}
      \sum_{k=0}^{-m_j-1}
      \chi_G\!\left(
        B_j,\mathcal O_{B_j}(P_j-kB_j)
      \right).
\end{align*}
\end{proposition}

\begin{remark}
The formula handles positive and negative coefficients at the same time, so no preliminary decomposition $D=D^+-D^-$ is needed. Different orderings change the intermediate restrictions, but not the final class in $\operatorname{Rep}(G)$.
\end{remark}

\subsection{Internal intersections of an orbit divisor}

Proposition~\ref{prop:ordered-peeling} is complete if the Euler characteristics on the reduced orbit divisors $B_j$ are left unchanged. To express them in terms of smooth components, however, one must include the intersections among components of the same orbit divisor.

Fix an orbit divisor
\[
    B=\bigcup_{a\in A}C_a
\]
and a $G$-linearized line bundle $F$ on $X$. For a nonempty subset $I\subseteq A$, put
\[
    C_I:=\bigcap_{a\in I}C_a,
\]
omitting empty intersections. Since $B$ has simple normal crossings, every $C_I$ is smooth of codimension $|I|$, although it may be disconnected.

There is a $G$-equivariant \v{C}ech resolution
\begin{equation}
\label{eq:snc-cech}
    0\longrightarrow F|_B
    \longrightarrow
    \bigoplus_{|I|=1}F|_{C_I}
    \longrightarrow
    \bigoplus_{|I|=2}F|_{C_I}
    \longrightarrow\cdots
    \longrightarrow
    \bigoplus_{|I|=\dim X}F|_{C_I}
    \longrightarrow0.
\end{equation}
To make the $G$-action transparent, we write it with alternating cochains. A section in degree $q$ assigns to each ordered $q$-tuple $(a_1,\ldots,a_q)$ of distinct elements of $A$ a section $s_{a_1\cdots a_q}$ of $F$ on $C_{\{a_1,\ldots,a_q\}}$, alternating under permutations of the indices. The differential is the usual alternating sum of restrictions. The group acts by
\[
    (g\cdot s)_{ga_1\cdots ga_q}=g\cdot s_{a_1\cdots a_q},
\]
using the linearization of $F$, and this action commutes with the differential. Choosing an order on each $I$ identifies the degree-$q$ term with $\bigoplus_{|I|=q}F|_{C_I}$. Under this identification, an element $g$ with $gI=I$ acts on the summand $F|_{C_I}$ through the linearization, multiplied by $\operatorname{sgn}(g|_I)$.

Let $Z$ be a connected component of a nonempty $C_I$, and let
\[
    G_{I,Z}
    :=\{g\in G:gI=I,\ gZ=Z\}
\]
be the isotropy subgroup of the pair $(I,Z)$. Its permutation action on $I$ defines a sign character
\[
    \varepsilon_I(g):=\operatorname{sgn}(g|_I).
\]
The resolution \eqref{eq:snc-cech} gives the following formula.

\begin{proposition}[Equivariant SNC inclusion--exclusion]
\label{prop:snc-inclusion-exclusion}
For every $G$-linearized line bundle $F$ on $X$,
\begin{equation}
\label{eq:snc-euler}
    \chi_G(B,F|_B)
    =\sum_{q\geq1}(-1)^{q-1}
      \sum_{\substack{[(I,Z)]\\|I|=q}}
      \operatorname{Ind}_{G_{I,Z}}^G
      \left(
        \varepsilon_I\otimes
        \chi_{G_{I,Z}}(Z,F|_Z)
      \right),
\end{equation}
where the inner sum runs over the $G$-orbits of pairs $(I,Z)$.
\end{proposition}

\begin{proof}
Exactness of \eqref{eq:snc-cech} can be checked on completed stalks. At a point $x$, choose local coordinates in which the components of $B$ through $x$ are $C_a=\{z_a=0\}$ for $a$ in a subset $A_x\subseteq A$. The other components do not meet a neighborhood of $x$. Since $F$ is a line bundle, it suffices to treat $F=\mathcal O_X$. All maps preserve the grading by monomials $z^m$. A monomial $z^m$ is nonzero on $C_I$ exactly when $I\subseteq J_m:=\{a\in A_x:m_a=0\}$, and it is nonzero on $B$ exactly when $J_m\neq\emptyset$. Hence the $z^m$-graded piece of \eqref{eq:snc-cech} vanishes if $J_m=\emptyset$. Otherwise it is the augmented cochain complex of the full simplex on $J_m$, which is exact.

The equivariant Euler characteristic is additive on exact sequences of $G$-equivariant coherent sheaves, so
\[
    \chi_G(B,F|_B)
    =\sum_{q\geq1}(-1)^{q-1}
      \chi_G\Bigl(\bigoplus_{|I|=q}F|_{C_I}\Bigr).
\]
The degree-$q$ term is the direct sum of the sheaves $F|_Z$ over the pairs $(I,Z)$ with $|I|=q$. The group permutes these summands, and $G_{I,Z}$ is the stabilizer of $(I,Z)$. Hence the cohomology of the sum over one $G$-orbit of pairs is induced from the cohomology of a single summand, viewed as a $G_{I,Z}$-representation. By the description of the action above, $G_{I,Z}$ acts on $F|_Z$ through the linearization twisted by $\varepsilon_I$. Summing over the orbits gives \eqref{eq:snc-euler}.
\end{proof}

Only intersections among components of the same orbit divisor $B_j$ occur in \eqref{eq:snc-euler}. Intersections with other orbit divisors are already encoded in the restricted line bundle $\mathcal O_{B_j}(P_j\pm kB_j)$ and must not be counted again.

\subsection{Surface case and internal nodes}

Assume now that $X$ is a surface. An SNC orbit divisor has only pairwise transverse intersections, so the preceding inclusion--exclusion formula reduces to the equivariant normalization sequence.

Let
\[
    \mathcal N_j:=\operatorname{Sing}(B_j)
\]
be the set of internal nodes of $B_j$, namely the intersection points of two distinct components of that orbit divisor. For $p\in\mathcal N_j$, let $G_p$ denote its isotropy subgroup. The action of $G_p$ on the two local branches defines the sign character
\[
    \varepsilon_p\colon G_p\longrightarrow\{\pm1\}.
\]

For a $G$-linearized line bundle $F$ on $X$, Proposition~\ref{prop:snc-inclusion-exclusion} gives
\begin{equation}
\label{eq:surface-normalization}
\begin{aligned}
    \chi_G(B_j,F|_{B_j})
    ={}&\operatorname{Ind}_{H_j}^G
       \chi_{H_j}(C_j,F|_{C_j})\\
     &-\sum_{[p]\in\mathcal N_j/G}
       \operatorname{Ind}_{G_p}^G
       \bigl(F_p\otimes\varepsilon_p\bigr).
\end{aligned}
\end{equation}
Indeed, only $|I|\leq2$ occurs. The pairs with $|I|=1$ form one $G$-orbit, with stabilizer $H_j$ and trivial sign. A pair with $|I|=2$ is determined by a node $p$, and its stabilizer is $G_p$, since $G_p$ permutes the two components through $p$. Here $F_p$ is the one-dimensional representation of $G_p$ on the fiber
of $F$ at $p$. Equivalently, \eqref{eq:surface-normalization} comes from the equivariant normalization sequence. Locally, the last map in the normalization sequence takes
the difference of the values on the two branches. Exchanging them
changes its sign, which explains $\varepsilon_p$. Thus a node term
depends both on the branch permutation and on the fiber character in
the relevant peeling layer. All restrictions retain their induced
linearizations. In particular, $\mathcal O_X(C_j)|_{C_j}$ carries the
induced normal-bundle linearization.

For compact notation, set
\begin{equation}
\label{eq:curve-node-contribution}
    \mathcal C_j(F):=\chi_G(B_j,F|_{B_j}),
\end{equation}
computed by \eqref{eq:surface-normalization}.

Combining \eqref{eq:surface-normalization} with Proposition~\ref{prop:ordered-peeling} gives the surface formula.

\begin{theorem}[Surface divisor reduction formula]
\label{thm:divisor-reduction}
Let $X$ be a smooth projective surface over $\mathbb C$ with an action of a finite group $G$, and let
\[
    D=\sum_{j=1}^r m_jB_j
\]
be an ordered $G$-invariant divisor with SNC support as above. Then
\begin{align*}
    \chi_G(X,\mathcal O_X(D))
    ={}&\chi_G(X,\mathcal O_X)\\
    &+\sum_{\substack{1\leq j\leq r\\m_j>0}}
      \sum_{k=1}^{m_j}
      \mathcal C_j\!\left(
        \mathcal O_X(P_j+kB_j)
      \right)\\
    &-\sum_{\substack{1\leq j\leq r\\m_j<0}}
      \sum_{k=0}^{-m_j-1}
      \mathcal C_j\!\left(
        \mathcal O_X(P_j-kB_j)
      \right).
\end{align*}
\end{theorem}

In a positive layer, the internal-node terms enter with a minus sign through $\mathcal C_j$. In a negative layer, they acquire a plus sign after the outer minus sign is distributed. The point terms are explicit representations determined by the fiber linearizations and the branch-permutation characters.

The curve terms may be evaluated using Arapura's Chevalley--Weil formula
\cite[Theorem~2.4]{arapura_residues_2022} when $H_j$ acts effectively on
$C_j$. Otherwise, one can compute the character element by element.
If $h\in H_j$ acts trivially on $C_j$, it acts on $F|_{C_j}$ by a constant
scalar, and its trace on the Euler characteristic is that scalar times
$\deg(F|_{C_j})+1-g(C_j)$. For any other $h$, restrict to $\langle h\rangle$.
The kernel of its action on $C_j$ acts on the line bundle by a character.
Extend this character to the cyclic group $\langle h\rangle$ and twist
by its inverse. The twisted linearization factors through the effective
cyclic quotient, where Arapura's formula applies. Twisting back recovers
the required character value. Thus no fiber character is lost by passing
to an effective action.

\subsection{Higher-dimensional form}

In arbitrary dimension, combine Proposition~\ref{prop:ordered-peeling}
with Proposition~\ref{prop:snc-inclusion-exclusion}, using
$F=\mathcal O_X(P_j+kB_j)$ in a positive layer and
$F=\mathcal O_X(P_j-kB_j)$ in a negative layer. This reduces
$\chi_G(X,\mathcal O_X(D))$ to $\chi_G(X,\mathcal O_X)$ and equivariant
Euler characteristics on smooth strata of smaller dimension. Curve
strata are treated as above. Higher-dimensional strata require further
reduction or a suitable equivariant index formula.

%% file: ch-obstruction.tex
\ProvidesFile{ch-obstruction.tex}

% The diagrams below require \usepackage{tikz}.

\section{The Quotient-Resolution Method and the Logarithmic Obstruction}
\label{sec:obstruction}

After the divisor-peeling reduction of Section~\ref{sec:reduction}, it remains to compute the equivariant Euler characteristic of the structure sheaf. The most direct surface analogue of Arapura's method is to take the quotient, pass to a resolution, and compute the Chern classes of the multiplicity bundles from the residues of their induced connections.

This method works directly for singularities of type $A_n$. For the $D_4$ singularity, however, the induced connection on the natural extension is not logarithmic. The explicit pole exhibited below explains the obstruction and motivates the pre-quotient blow-up method of Section~\ref{sec:blowup}.

\subsection{Natural extensions on a quotient resolution}

Let $X$ be a smooth projective surface with an effective action of a finite group $G$, and let
\[
    \pi\colon X\longrightarrow Y:=X/G
\]
be the quotient map. Fix an irreducible character $\xi\in\operatorname{Irr}(G)$ and set
\[
    \mathcal F_\xi
    :=\operatorname{Hom}_G(W_\xi,\pi_*\mathcal O_X).
\]
As in Section~\ref{sec:preliminaries}, this is a reflexive sheaf of rank $\xi(1)$ on $Y$, and
\begin{equation}
\label{eq:obstruction-isotypic-euler}
    \chi(Y,\mathcal F_\xi)
    =\left\langle
        \chi_G(X,\mathcal O_X),\xi
      \right\rangle_G.
\end{equation}
The full isotypic summand is $W_\xi\otimes\mathcal F_\xi$ and has rank
$\xi(1)^2$. Working with $\mathcal F_\xi$ gives the multiplicity directly.

Let
\[
    \mu\colon S\longrightarrow Y
\]
be a resolution that is an isomorphism over $Y_{\mathrm{reg}}$. Since quotient singularities are rational, Proposition~\ref{extension_prop} applies to the natural extension
\[
    \mathcal V_\xi^{\mathrm{nat}}
    :=\mu^*\mathcal F_\xi/\operatorname{torsion}.
\]
Thus $\mathcal V_\xi^{\mathrm{nat}}$ is locally free and
\begin{equation}
\label{eq:obstruction-natural-pushforward}
    \mu_*\mathcal V_\xi^{\mathrm{nat}}
    \cong\mathcal F_\xi,
    \qquad
    R^i\mu_*\mathcal V_\xi^{\mathrm{nat}}=0
    \quad(i>0).
\end{equation}
Consequently,
\begin{equation}
\label{eq:obstruction-natural-euler}
    \chi(S,\mathcal V_\xi^{\mathrm{nat}})
    =\chi(Y,\mathcal F_\xi).
\end{equation}

Let $B\subset Y$ be the reduced divisorial branch locus and put
$Y^\circ=Y_{\mathrm{reg}}\setminus B$. On $Y^\circ$, the trivial connection
\[
    d\colon\mathcal O_X\longrightarrow\Omega_X^1
\]
induces flat connections on the direct-image sheaf and its multiplicity
sheaves. Assume that the exceptional divisor together with the strict
transform of $B$ is SNC. Regular singularity then gives a locally free
logarithmic extension to $S$. The issue is whether the induced meromorphic
connection is logarithmic on the natural extension
$\mathcal V_\xi^{\mathrm{nat}}$ itself.

In the two local models below, the isotropy subgroup acts freely off the
origin, so the connection is defined on the punctured quotient and only
the exceptional divisor is needed. An isolated quotient singularity by
itself does not exclude divisorial ramification.

\subsection{The cyclic case: singularities of type \texorpdfstring{$A_n$}{An}}

Set $N:=n+1$, with $n\geq1$. Let $p\in X$ have cyclic isotropy subgroup
\[
    G_p=\langle g\rangle\cong\mathbb Z_N,
\]
acting on the coordinate functions $(x,y)$ by
\[
    g\cdot x=\zeta_Nx,\qquad g\cdot y=\zeta_N^{-1}y,
    \qquad \zeta_N=\exp\!\left(\frac{2\pi i}{N}\right).
\]
Thus $\chi(g)=\zeta_N$ defines the character of $x$. As in
Sections~\ref{sec:blowup} and~\ref{sec:general-hj}, these are function
characters for $(g\cdot h)(z)=h(g^{-1}z)$.
The quotient has an $A_n$ singularity because
\[
    \mathbb C[x,y]^{G_p}
    \cong
    \frac{\mathbb C[a,b,c]}{(ab-c^N)},
    \qquad
    a=x^N,\quad b=y^N,\quad c=xy.
\]
Its minimal resolution has exceptional chain
\[
    F_1,\ldots,F_n,
\]
ordered from the strict transform of $x=0$ to that of $y=0$, as in
Section~\ref{sec:general-hj}. The intersections are
\[
    F_i^2=-2,
    \qquad
    F_i\cdot F_{i+1}=1,
    \qquad
    F_i\cdot F_j=0
    \quad(|i-j|>1).
\]

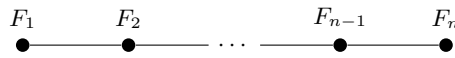
\begin{figure}[ht]
\centering
\begin{tikzpicture}[
    vertex/.style={circle,fill=black,inner sep=1.8pt},
    every node/.style={font=\small}
]
\node[vertex,label=above:{$F_1$}] (E1) at (0,0) {};
\node[vertex,label=above:{$F_2$}] (E2) at (1.4,0) {};
\node (dots) at (2.8,0) {$\cdots$};
\node[vertex,label=above:{$F_{n-1}$}] (Enm) at (4.2,0) {};
\node[vertex,label=above:{$F_n$}] (En) at (5.6,0) {};

\draw (E1)--(E2);
\draw (E2)--(dots);
\draw (dots)--(Enm);
\draw (Enm)--(En);
\end{tikzpicture}
\caption{The exceptional chain in the minimal resolution of an $A_n$ singularity, shown schematically for $n\geq4$. For smaller $n$, use the corresponding shorter chain.}
\label{fig:An-exceptional-chain}
\end{figure}

For $0\leq i\leq n$, the resolution charts have coordinates
\[
    A=\frac{a}{c^{n-i}}
      =\frac{x^{i+1}}{y^{n-i}},
    \qquad
    B=\frac{b}{c^i}
      =\frac{y^{n-i+1}}{x^i}.
\]
For $1\leq i\leq n-1$, one has $F_i=\{A=0\}$ and
$F_{i+1}=\{B=0\}$. At $i=0$, $A=0$ is the strict transform of $x=0$
and $F_1=\{B=0\}$. At $i=n$, $F_n=\{A=0\}$ and $B=0$ is the
strict transform of $y=0$. The pullbacks are
\[
    a=A^{N-i}B^{n-i},\qquad
    b=A^iB^{i+1},\qquad c=AB.
\]
Differentiating gives
\[
    \begin{pmatrix}
        d\log A\\
        d\log B
    \end{pmatrix}
    =
    \begin{pmatrix}
        i+1 & -(n-i)\\
        -i & n-i+1
    \end{pmatrix}
    \begin{pmatrix}
        d\log x\\
        d\log y
    \end{pmatrix}.
\]
The determinant is $N$, and hence
\begin{equation}
\label{eq:An-dlog-inversion}
\begin{aligned}
    d\log x
    &=\frac{n-i+1}{N}\,d\log A
      +\frac{n-i}{N}\,d\log B,\\
    d\log y
    &=\frac{i}{N}\,d\log A
      +\frac{i+1}{N}\,d\log B.
\end{aligned}
\end{equation}

For $0\leq j<N$, the character module for $\chi^j$ in the local $G_p$-quotient model is
generated over the invariant ring by $x^j$ and $y^{N-j}$.
On the chart above,
\[
    \frac{x^j}{y^{N-j}}=A^{j-i}B^{j-i-1}.
\]
Thus $x^j$ is a frame when $j\leq i$, and $y^{N-j}$ is a frame when
$j\geq i+1$. The natural extension of the local direct-image sheaf
therefore has the frame
\[
    1,x,\ldots,x^i,y,\ldots,y^{n-i},
\]
with an empty range omitted. Using \eqref{eq:An-dlog-inversion}, the connection induced by $d$ satisfies
\begin{align*}
    \nabla(x^j)
    &=x^j\otimes
    \left(
        \frac{(n-i+1)j}{N}\,d\log A
        +\frac{(n-i)j}{N}\,d\log B
    \right),\\
    \nabla(y^k)
    &=y^k\otimes
    \left(
        \frac{ik}{N}\,d\log A
        +\frac{(i+1)k}{N}\,d\log B
    \right).
\end{align*}
The connection matrix is diagonal and logarithmic on every chart,
including the two end charts. Its residues along the strict transforms of
$x=0$ and $y=0$ are zero.

For $0\leq j\leq n$, the natural extension of the $\chi^j$-character
module is a line bundle. Its residue along $F_i$ is
\begin{equation}
\label{eq:An-residues}
    R_i^{(j)}=
    \begin{cases}
        \displaystyle \frac{(N-j)i}{N},
        &1\leq i\leq j,\\[0.8em]
        \displaystyle
        \frac{j(N-i)}{N},
        &j\leq i\leq n.
    \end{cases}
\end{equation}
Here $1\leq i\leq n$. The two expressions agree at $i=j$, and
$R_i^{(0)}=0$. This is $R_{i,\chi^j}$ in the notation of
Section~\ref{sec:general-hj}.

\begin{proposition}[Local correction for an $A_n$ singularity]
\label{prop:An-local-correction}
For $0\leq j\leq n$, the correction attached to the local character $\chi^j$ is
\begin{equation}
\label{eq:An-character-correction}
    m_{p,\chi^j}
    =\frac{j(N-j)}{2N}.
\end{equation}
Equivalently, the local correction class function is
\begin{equation}
\label{eq:An-local-character-class}
    \rho_{A_n}
    =\sum_{j=0}^{N-1}
      \frac{j(N-j)}{2N}\,\chi^j
    \in
    \operatorname{Rep}(G_p)\otimes_{\mathbb Z}\mathbb Q.
\end{equation}
For $\xi\in\operatorname{Irr}(G)$, the contribution of the orbit of $p$ is
\begin{equation}
\label{eq:An-global-pairing}
    m_{p,\xi}
    =\left\langle
        \rho_{A_n},
        \operatorname{Res}_{G_p}^G\xi
      \right\rangle_{G_p}.
\end{equation}

The same class function can also be written in terms of the symmetric-power characters
\[
    \tau_{p,k}
    :=\operatorname{char}\!\left(
        \operatorname{Sym}^k(T_{X,p}^*)
      \right),
    \qquad
    \tau_{p,k}=0\quad(k<0),
\]
as
\begin{equation}
\label{eq:An-class-function}
    \rho_{A_n}
    =\sum_{k=1}^{N-1}
      \frac{k(N-k)}{4N}
      \bigl(\tau_{p,k}-\tau_{p,k-2}\bigr).
\end{equation}
\end{proposition}

\begin{proof}
The $A_n$ resolution is crepant, so
\[
    K_S\cdot F_i=0
    \qquad(1\leq i\leq n).
\]
Retain the terms of the global residue formula assigned to this chain.
The boundary residues vanish, so the correction coefficient is
\[
    m_{p,\chi^j}
    =\sum_{i=1}^{n}\bigl(R_i^{(j)}\bigr)^2
      -\sum_{i=1}^{n-1}R_i^{(j)}R_{i+1}^{(j)}.
\]
For $j=0$, all residues vanish. For $1\leq j\leq n$, put
$R_0^{(j)}=R_N^{(j)}=0$. By \eqref{eq:An-residues},
$2R_i^{(j)}-R_{i-1}^{(j)}-R_{i+1}^{(j)}$ is $1$ at $i=j$ and
zero elsewhere. Hence
\[
\begin{aligned}
    m_{p,\chi^j}
    &=\frac12\sum_{i=1}^{n}R_i^{(j)}
       \bigl(2R_i^{(j)}-R_{i-1}^{(j)}-R_{i+1}^{(j)}\bigr)\\
    &=\frac12R_j^{(j)}
     =\frac{j(N-j)}{2N}.
\end{aligned}
\]
This proves \eqref{eq:An-character-correction} and
\eqref{eq:An-local-character-class}. Locally, the natural extension of
$\mathcal F_\xi$ splits into these character line bundles with the
multiplicities in $\operatorname{Res}_{G_p}^G\xi$. Frobenius reciprocity
then gives \eqref{eq:An-global-pairing}.

For the symmetric-power form, the cotangent representation has characters $\chi$ and $\chi^{-1}$, so
\[
    \tau_{p,k}
    =\sum_{a=0}^{k}\chi^{k-2a}.
\]
For $k\geq1$,
\[
    \tau_{p,k}-\tau_{p,k-2}
    =\chi^k+\chi^{-k}.
\]
For $1\leq j<N$, the terms indexed by $k=j$ and $k=N-j$ in \eqref{eq:An-class-function} give the coefficient
\[
    2\cdot\frac{j(N-j)}{4N}
    =\frac{j(N-j)}{2N}
\]
of $\chi^j$. If $2j=N$, the two characters coincide and occur with multiplicity two. The coefficient of the trivial character is zero. Thus \eqref{eq:An-class-function} agrees with \eqref{eq:An-local-character-class}.
\end{proof}

\begin{remark}
The formula in local irreducible characters, \eqref{eq:An-local-character-class}, is the primary form used later. Equation~\eqref{eq:An-class-function} is a useful geometric reformulation in terms of the cotangent representation.
\end{remark}

\subsection{The \texorpdfstring{$D_4$}{D4} obstruction}

We next consider the binary dihedral group
\[
    G_p\cong\operatorname{Dic}_2\cong Q_8
\]
acting on coordinate functions by
\[
    g=
    \begin{pmatrix}
        i&0\\
        0&-i
    \end{pmatrix},
    \qquad
    h=
    \begin{pmatrix}
        0&1\\
        -1&0
    \end{pmatrix}.
\]
Its invariant ring is
\[
    \mathbb C[x,y]^{G_p}
    \cong
    \frac{\mathbb C[a,b,c]}
          {(a^2-b^2c+4c^3)},
\]
where
\[
    a=xy(x^4-y^4),
    \qquad
    b=x^4+y^4,
    \qquad
    c=x^2y^2.
\]
The quotient has a $D_4$ singularity. Its minimal resolution has a central $(-2)$-curve $T$ meeting three disjoint outer $(-2)$-curves $F_1,F_2,F_3$.

\begin{figure}[htbp]
\centering
\begin{tikzpicture}[
    vertex/.style={circle,fill=black,inner sep=2.2pt},
    every node/.style={font=\small}
]
    \node[vertex,label=below:$T$] (E1) at (0,0) {};
    \node[vertex,label=below:$F_1$] (E2) at (-1.8,0) {};
    \node[vertex,label=below:$F_2$] (E3) at (1.8,0) {};
    \node[vertex,label=right:$F_3$] (E4) at (0,1.5) {};
    \draw (E1)--(E2);
    \draw (E1)--(E3);
    \draw (E1)--(E4);
\end{tikzpicture}
\caption{The exceptional curves in the minimal resolution of a $D_4$ singularity.}
\label{fig:D4-dual-graph}
\end{figure}
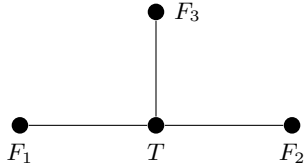

Near $T\cap F_1$, use coordinates
\[
    A=\frac{a}{c}
      =\frac{x^4-y^4}{xy},
    \qquad
    C=\frac{c}{b}
      =\frac{x^2y^2}{x^4+y^4},
\]
with $T=\{A=0\}$ and $F_1=\{C=0\}$. These coordinates come from
blowing up the origin and then resolving the node with $c/b=0$.
The pullbacks of the invariant functions are
\[
    a=\frac{A^3C^2}{1-4C^2},\qquad
    b=\frac{A^2C}{1-4C^2},\qquad
    c=\frac{A^2C^2}{1-4C^2}.
\]
We work near $(A,C)=(0,0)$, where $1-4C^2$ is a unit.
A direct calculation gives
\begin{align}
\label{eq:D4-dlog-x}
    d\log x
    &=\frac12d\log A
      +\frac{1}{2(1-4C^2)}d\log C
      -\frac{xy}{4AC}d\log C,\\
\label{eq:D4-dlog-y}
    d\log y
    &=\frac12d\log A
      +\frac{1}{2(1-4C^2)}d\log C
      +\frac{xy}{4AC}d\log C.
\end{align}
This chart is enough to exhibit the obstruction.

The group $Q_8$ has four irreducible characters of degree one and one
irreducible character $\theta$ of degree two. In this local quotient model, the multiplicity bundle
$\mathcal V_\theta^{\mathrm{nat}}$ has rank two. For the calculation,
we use the full $\theta$-isotypic extension
$W_\theta\otimes\mathcal V_\theta^{\mathrm{nat}}$, which has rank
\[
    \theta(1)^2=4
\]
and has the local frame
\[
    x,\quad y,\quad y^3,\quad x^3.
\]
To check that these form a frame of the natural extension, note that
the $\theta$-isotypic module consists of the polynomials of odd total
degree. The central element $-I$ acts by $-1$ precisely on this isotypic
summand. On our chart, the relations
\[
    x^4=\frac{b+Axy}{2},\qquad
    y^4=\frac{b-Axy}{2},\qquad x^2y^2=c
\]
reduce every odd monomial to an $\mathcal O_S$-linear combination of
$x,y,x^3,y^3,x^2y,xy^2$. The remaining two monomials satisfy
\[
    x^2y=ACx+2Cy^3,\qquad
    xy^2=-ACy+2Cx^3.
\]
Thus the four displayed sections generate the natural extension.
It is locally free of rank four by Proposition~\ref{extension_prop},
so they form a frame.

Write $\nabla$ for the connection on this full isotypic extension.
Equation~\eqref{eq:D4-dlog-x} gives
\[
    \nabla(x)
    =x\otimes
    \left(
        \frac12d\log A
        +\frac{1}{2(1-4C^2)}d\log C
    \right)
    -x^2y\otimes\frac{1}{4AC}d\log C.
\]
Substituting the relation for $x^2y$, we obtain
\begin{equation}
\label{eq:D4-pole-obstruction}
\begin{aligned}
    \nabla(x)
    ={}&x\otimes
    \left(
        \frac12d\log A
        +\frac{1+4C^2}{4(1-4C^2)}d\log C
    \right)\\
    &-y^3\otimes\frac{1}{2A}d\log C.
\end{aligned}
\end{equation}

\begin{proposition}[Failure of logarithmicity for $D_4$]
\label{prop:D4-log-obstruction}
The induced connection on the natural multiplicity bundle
$\mathcal V_\theta^{\mathrm{nat}}$ is not logarithmic along the
exceptional divisor. The same holds for the full $\theta$-isotypic extension.
\end{proposition}

\begin{proof}
Near $T\cap F_1$, the sheaf of logarithmic one-forms is freely generated by
\[
    d\log A=\frac{dA}{A},
    \qquad
    d\log C=\frac{dC}{C}.
\]
A logarithmic connection matrix in a holomorphic frame must have
holomorphic coefficients in this basis. The coefficient of $y^3$ in
\eqref{eq:D4-pole-obstruction} is $-\frac{1}{2A}d\log C$, and
\[
    \frac1A d\log C
    \notin
    \Omega_S^1\bigl(\log(T+F_1)\bigr).
\]
Since $y^3$ is a member of the frame, this pole proves that the
connection on the full isotypic extension is not logarithmic. That
connection is the identity on $W_\theta$ tensored with the connection
on $\mathcal V_\theta^{\mathrm{nat}}$, so the multiplicity connection
cannot be logarithmic either.
\end{proof}

The flat connection on the punctured quotient is still regular singular,
so it admits a locally free logarithmic extension. The natural extension,
whose Euler characteristic equals that of the reflexive multiplicity
sheaf on the quotient, is not such an extension.

For $D_4$, replacing the frame by $x,y,y^3/A,x^3/A$ gives a logarithmic
extension on this chart. The displayed formula handles $x$, and a direct
calculation gives
\[
\begin{aligned}
    \nabla(y^3/A)
    ={}&\frac{y^3}{A}\otimes
       \left(\frac12d\log A
          +\frac{3(1+4C^2)}{4(1-4C^2)}d\log C\right)\\
       &+x\otimes\frac{3C^2}{2(1-4C^2)}d\log C.
\end{aligned}
\]
The other two frame vectors follow by $Q_8$-equivariance. A global
calculation would then require checking the transition functions on all
resolution charts and computing the Euler characteristic of the difference
between the two extensions, as in \eqref{eq:prelim-extension-comparison}.
The same bookkeeping becomes much harder for larger non-abelian isotropy subgroups.

We therefore use the $D_4$ calculation only to show the obstruction.
In Section~\ref{sec:blowup}, we instead blow up the fixed point before
taking the quotient. The isotropy subgroups along the exceptional curve
then act diagonally, and after removing quasi-reflections the remaining
quotient singularities are cyclic of Hirzebruch--Jung type.

%% file: ch-blowup.tex
\ProvidesFile{ch-blowup.tex}

% The diagrams below require:
% \usepackage{tikz,tikz-cd}
% \usetikzlibrary{positioning}

\section{The Pre-Quotient Blow-Up Method}
\label{sec:blowup}

Section~\ref{sec:obstruction} showed why resolving the quotient first can be difficult for a non-cyclic isotropy subgroup: the connection on the natural extension need not be logarithmic. An elementary modification can sometimes repair the extension, but such a construction depends on detailed local coordinates and becomes difficult for larger groups.

We now reverse the order of operations. We first blow up the fixed point on the smooth surface and then take the quotient. The exceptional curve records the projective action of the original isotropy subgroup. At its special points, the new isotropy subgroups are diagonal, and after removing quasi-reflections the remaining quotient singularities are cyclic. This reduces the non-abelian local problem to a collection of cyclic calculations. Two results used below are proved for general diagonal actions in Section~\ref{sec:general-hj}: the natural extensions are logarithmic, and the local correction does not depend on the model used to compute it.

\subsection{Blowing up before taking the quotient}

Let $(X,p)$ be a smooth complex surface germ with an effective action of a finite group
\[
    G_p\subset \operatorname{GL}(T_{X,p})
\]
fixing $p$. We identify the isotropy subgroup with its faithful tangent representation. Let
\[
    f\colon (X',E)\longrightarrow (X,p)
\]
be the blow-up at $p$. Then
\[
    E=f^{-1}(p)=\mathbb P(T_{X,p})\cong\mathbb P^1,
    \qquad E^2=-1.
\]
The action of $G_p$ lifts uniquely to $X'$, and $E$ is $G_p$-invariant.

Globally, if $p$ is not fixed by all of $G$, one blows up the full orbit $G\cdot p$. Equivariant birational invariance gives
\begin{equation}
\label{eq:blowup-birational-invariance}
    \chi_G(X',\mathcal O_{X'})
    =\chi_G(X,\mathcal O_X).
\end{equation}
Thus the blow-up does not change the equivariant Euler characteristic. It replaces the local contribution at $p$ by contributions supported on the exceptional curve.

Throughout this section, characters label the action on coordinate functions:
\[
    (g\cdot h)(z)=h(g^{-1}z).
\]
They are the inverses of the corresponding tangent characters, as in
Section~\ref{sec:general-hj}. We write $\zeta_m=\exp(2\pi i/m)$.

\subsection{Isotropy subgroups on the exceptional curve}

For $q\in E$, let
\[
    G_q:=\{g\in G_p:gq=q\}
\]
be its isotropy subgroup. A point of $E=\mathbb P(T_{X,p})$ corresponds to a line in $T_{X,p}$. This gives the main local simplification.

\begin{lemma}[Isotropy subgroups after blowing up]
\label{lem:diag-formal}
Let $q\in E$.
\begin{enumerate}
    \item There are local holomorphic coordinates $(u,v)$ at $q$, with $E=\{u=0\}$, in which the action of $G_q$ is diagonal.

    \item The group $G_q$ is abelian. If $G_p\subset\operatorname{SL}_2(\mathbb C)$, then $G_q$ is cyclic.

    \item In general, let $H_q\trianglelefteq G_q$ be the subgroup generated by the quasi-reflections in this action. The quotient $X'/H_q$ is smooth near the image of $q$, and the residual group $G_q/H_q$ is cyclic. Hence $X'/G_q$ has, at worst, a cyclic Hirzebruch--Jung singularity at the image of $q$.
\end{enumerate}
\end{lemma}

\begin{proof}
The point $q$ corresponds to a line $L_q\subset T_{X,p}$ preserved by $G_q$. Since $G_q$ is finite, its complex representation on $T_{X,p}$ is semisimple. There is therefore a $G_q$-invariant complementary line $L'_q$ such that
\[
    T_{X,p}=L'_q\oplus L_q.
\]
After equivariant local linearization, choose linear coordinates $(z_1,z_2)$ with $L_q=\{z_1=0\}$ and $L'_q=\{z_2=0\}$. Then
\[
    g\cdot z_1=\mu(g)z_1,
    \qquad
    g\cdot z_2=\lambda(g)z_2
\]
for characters $\lambda,\mu\colon G_q\to\mathbb C^*$ on coordinate functions.

On the blow-up chart containing $q$, put
\[
    u=z_2,
    \qquad
    v=\frac{z_1}{z_2}.
\]
Then $E=\{u=0\}$ and
\begin{equation}
\label{eq:diag-matrix-coords}
    g\cdot(u,v)
    =\bigl(\lambda(g)u,\mu(g)\lambda(g)^{-1}v\bigr).
\end{equation}
This proves the first assertion. The action on this chart is faithful, so $G_q$ is a finite subgroup of the diagonal torus $(\mathbb C^*)^2$ and is abelian.

Now suppose that $G_p\subset\operatorname{SL}_2(\mathbb C)$. The determinant of the action on coordinate functions is also one, so
\[
    \mu(g)=\lambda(g)^{-1}.
\]
Moreover, the character $g\mapsto\lambda(g)$ is injective on $G_q$: if $\lambda(g)=1$, then also $\mu(g)=1$, so $g$ acts trivially on $T_{X,p}$ and hence is the identity in the effective finite action. Thus $G_q$ embeds in $\mathbb C^*$ and is cyclic.

For the final assertion, let $H_1$ and $H_2$ be the kernels of the characters of $G_q$ on $v$ and on $u$, respectively. By \eqref{eq:diag-matrix-coords}, $H_1$ acts only on $u$ and fixes $E$ pointwise, while $H_2$ acts only on $v$. The nonidentity elements of $H_1\cup H_2$ are exactly the quasi-reflections, so $H_q=H_1H_2$. Effectiveness gives $H_1\cap H_2=\{1\}$, and $H_q$ is normal because $G_q$ is abelian. With $r_i:=|H_i|$, the quotient by $H_q$ is smooth, with coordinates $u^{r_1}$ and $v^{r_2}$. If $g\in G_q$ acts trivially on $u^{r_1}$, then $\lambda(g)$ is an $r_1$-th root of unity. Since $\lambda|_{H_1}$ is faithful, $\lambda(g)=\lambda(h)$ for some $h\in H_1$. Then $gh^{-1}\in H_2$, so $g\in H_q$. Thus $G_q/H_q$ acts faithfully on $u^{r_1}$, embeds in $\mathbb C^*$, and is cyclic. Its quotient is a cyclic Hirzebruch--Jung singularity.
\end{proof}

\subsection{Binary polyhedral isotropy subgroups}

Assume from now on that
\[
    G_p\subset\operatorname{SL}_2(\mathbb C)
\]
is a non-cyclic binary polyhedral group. The original quotient has a rational double point of type $D$ or $E$. A nonidentity element of this finite group has no eigenvalue $1$, so the action is free off $p$ near $p$. The exceptional configuration therefore gives the full local correction.

Every binary polyhedral group contains the central subgroup
\[
    H:=\{\pm I\}.
\]
This is the subgroup $H_E$ of Section~\ref{sec:general-hj}, acting trivially on $E$. On the blow-up, $-I$ acts by multiplication by $-1$ in the normal direction. It is therefore a quasi-reflection with fixed divisor $E$.

The effective action on $E$ is the action of
\[
    \overline G_p:=G_p/H\subset\operatorname{PGL}_2(\mathbb C).
\]
This is a dihedral, tetrahedral, octahedral, or icosahedral rotation group. We factor the quotient as
\begin{equation}
\label{eq:two-step-quotient}
\begin{tikzcd}[column sep=large]
    X'
    \arrow[r,"\pi_H"]
    \arrow[rr,bend left=22,"\pi'"]
    & Z:=X'/H
    \arrow[r,"\overline\pi"]
    & Y':=X'/G_p.
\end{tikzcd}
\end{equation}

Near a point of $E$, choose coordinates as in Lemma~\ref{lem:diag-formal}. The involution acts by
\[
    (u,v)\longmapsto(-u,v).
\]
Thus $Z$ is smooth near the image of $E$, with coordinates
\[
    w=u^2,
    \qquad v,
\]
and the image $E_Z$ of $E$ is $\{w=0\}$.

Let $q\in E$ have isotropy subgroup $G_q\cong\mathbb Z_{2k}$. Choose the generator $g$ so that its action on coordinate functions is
\begin{equation}
\label{eq:cyclic-action-before-reflection-quotient}
    g\cdot u=\zeta_{2k}u,
    \qquad
    g\cdot v=\zeta_{2k}^{-2}v.
\end{equation}
The central involution is $g^k$. On $Z$, the residual group $G_q/H\cong\mathbb Z_k$ is generated by the image $\overline g$ of $g$ and acts by
\[
    \overline g\cdot(w,v)
    =\bigl(\zeta_k w,\zeta_k^{-1}v\bigr).
\]
Hence the image of $q$ in $Y'$ has type $A_{k-1}$. For $k=1$, this is a smooth point.

Equivalently,
\[
    \mathbb C[u,v]^{G_q}
    =\mathbb C[u^{2k},v^k,u^2v]
    \cong
    \frac{\mathbb C[a,b,c]}{(ab-c^k)}.
\]
Here $a=u^{2k}$, $b=v^k$, and $c=u^2v$. The quotient map also has divisorial ramification along the image of $E$.

\subsection{The three special orbits}

A non-cyclic finite subgroup of $\operatorname{PGL}_2(\mathbb C)$ has exactly three orbits on $\mathbb P^1$ with nontrivial isotropy. For a binary polyhedral group, the inverse images of the corresponding isotropy subgroups in $G_p$ are cyclic of even order. These are the orbits on $E$ with isotropy subgroup larger than $H$.

\begin{table}[htbp]
\centering
\caption{Special orbits on $E=\mathbb P(T_{X,p})$.}
\label{tab:exceptional-orbits}
\renewcommand{\arraystretch}{1.2}
\small
\begin{tabular}{c|c|c|c|c}
\hline
Type & $G_p$ & $\overline G_p$ & Orbit sizes & Isotropy subgroups in $G_p$ \\
\hline
$D_{n+2}$
& $\operatorname{Dic}_n$
& $D_n$
& $2,\ n,\ n$
& $\mathbb Z_{2n},\ \mathbb Z_4,\ \mathbb Z_4$ \\
$E_6$
& $2\mathbf T$
& $A_4$
& $4,\ 6,\ 4$
& $\mathbb Z_6,\ \mathbb Z_4,\ \mathbb Z_6$ \\
$E_7$
& $2\mathbf O$
& $S_4$
& $6,\ 12,\ 8$
& $\mathbb Z_8,\ \mathbb Z_4,\ \mathbb Z_6$ \\
$E_8$
& $2\mathbf I$
& $A_5$
& $12,\ 30,\ 20$
& $\mathbb Z_{10},\ \mathbb Z_4,\ \mathbb Z_6$ \\
\hline
\end{tabular}
\end{table}

Here $D_n$ has order $2n$ and $\operatorname{Dic}_n$ has order $4n$, with $n\geq2$.

Suppose that the three isotropy orders are $2k_1,2k_2,2k_3$. On the minimal resolution $S\to Y'$, the corresponding quotient points give three chains of lengths $k_1-1$, $k_2-1$, and $k_3-1$.
Each chain meets the strict transform $T$ of the image of $E$ at one end.

The curve $T$ has self-intersection $-2$. Indeed, $\pi_H$ is a double cover branched along $E$, so $E_Z^2=-2$.
The map $Z\to Y'$ has degree $|\overline G_p|$ and is generically unramified along $E_Z$, so the image of $E_Z$ on the singular quotient $Y'$ has rational self-intersection $-2/|\overline G_p|$. Resolving an $A_{k_s-1}$ point through which this curve passes decreases its self-intersection by $(k_s-1)/k_s$. Riemann--Hurwitz for the cover $E\to E/\overline G_p\cong\mathbb P^1$ gives
\[
    \frac1{k_1}+\frac1{k_2}+\frac1{k_3}
    =1+\frac2{|\overline G_p|}.
\]
Therefore
\[
\begin{aligned}
    T^2
    &=-\frac2{|\overline G_p|}
      -\sum_{s=1}^3\frac{k_s-1}{k_s}\\
    &=-2.
\end{aligned}
\]

The final exceptional configuration is the corresponding ADE tree. In the order of Table~\ref{tab:exceptional-orbits}, the arm lengths are
\[
\begin{array}{c|c}
\text{type} & \text{three arm lengths} \\
\hline
D_{n+2} & n-1,1,1 \\
E_6 & 2,1,2 \\
E_7 & 3,1,2 \\
E_8 & 4,1,2.
\end{array}
\]
Every component is a smooth rational curve of self-intersection $-2$. Thus $S$ is also the minimal resolution of the rational double point $X/G_p$: the morphism $Y'\to X/G_p$ contracts the image of $E$, and $S$ contains no $(-1)$-curves. The pre-quotient blow-up therefore gives a second extension of the multiplicity sheaves to the same surface considered in Section~\ref{sec:obstruction}. Unlike the extension there, it is logarithmic by Proposition~\ref{prop:general-hj-natural-residues}, and it has the same Euler characteristic by \eqref{eq:blowup-birational-invariance}.

The following schematic records the geometry used in the residue calculation.

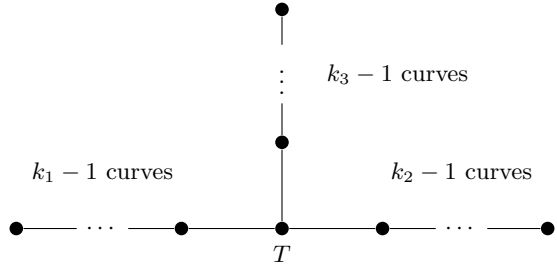
\begin{figure}[htbp]
\centering
\begin{tikzpicture}[
    scale=0.95,
    vertex/.style={circle,fill=black,inner sep=1.8pt},
    every node/.style={font=\small}
]
    \node[vertex,label=below:{$T$}] (T) at (0,0) {};

    \node[vertex] (A1) at (-1.4,0) {};
    \node (Adots) at (-2.5,0) {$\cdots$};
    \node[vertex] (Aend) at (-3.7,0) {};
    \draw (T)--(A1)--(Adots)--(Aend);
    \node[above=3mm of Adots] {$k_1-1$ curves};

    \node[vertex] (B1) at (1.4,0) {};
    \node (Bdots) at (2.5,0) {$\cdots$};
    \node[vertex] (Bend) at (3.7,0) {};
    \draw (T)--(B1)--(Bdots)--(Bend);
    \node[above=3mm of Bdots] {$k_2-1$ curves};

    \node[vertex] (C1) at (0,1.2) {};
    \node (Cdots) at (0,2.15) {$\vdots$};
    \node[vertex] (Cend) at (0,3.05) {};
    \draw (T)--(C1)--(Cdots)--(Cend);
    \node[right=3mm of Cdots] {$k_3-1$ curves};
\end{tikzpicture}
\caption{The star-shaped exceptional divisor obtained from the three special isotropy orbits. The arms are schematic. An arm of length one consists only of the curve next to $T$.}
\label{fig:ADE-star-resolution}
\end{figure}

\subsection{The cyclic arm calculation}
\label{subsec:cyclic-arm}

We now compute the contribution of one cyclic arm. Let
\[
    C_k=G_q=\langle g\rangle\cong\mathbb Z_{2k},
    \qquad
    \lambda(g)=\zeta_{2k},\qquad k\geq2,
\]
where $\lambda$ is the character on the normal coordinate function $u$. Thus
\[
    g\cdot u=\lambda(g)u,
    \qquad g\cdot v=\lambda(g)^{-2}v.
\]
The central subgroup $H=\langle g^k\rangle\cong\mathbb Z_2$ acts as a reflection along $E=\{u=0\}$.

After resolving the $A_{k-1}$ point, let $F_1,\ldots,F_{k-1}$ be the exceptional chain, indexed so that $F_1$ meets the strict transform $T$ of the reflection curve. The natural extension of the local direct-image sheaf splits into line bundles indexed by $\operatorname{Irr}(C_k)=\{1,\lambda,\ldots,\lambda^{2k-1}\}$.
These extensions are logarithmic, as proved generally in
Proposition~\ref{prop:general-hj-natural-residues}. For the character
$\lambda^m$, write $t_m$ for the residue along $T$ and
$R_i^{(m)}$ for the residue along $F_i$. These are
$t_{x,\lambda^m}$ and $R_{i,\lambda^m}$ in Section~\ref{sec:general-hj}.

For $0\leq r<k$, the module for $\lambda^{2r}$ is generated by $u^{2r}$ and $v^{k-r}$ over the invariant ring.
The module for $\lambda^{2r+1}$ is $u$ times this module. In the coordinates
$(w,v)=(u^2,v)$, the exceptional rays are
$((k-i)/k,i/k)$, ordered from $T$. Taking the minimum valuations of
these generators gives the following residues.

For $0\le r\le k-1$ and $1\le i\le k-1$, the residues for the even character $\lambda^{2r}$ are
\begin{equation}
\label{eq:even-arm-residues}
t_{2r}=0,
\qquad
R_i^{(2r)}=
\begin{cases}
\displaystyle \frac{(k-r)i}{k},
&1\le i\le r,\\[0.7em]
\displaystyle \frac{r(k-i)}{k},
&r\le i\le k-1.
\end{cases}
\end{equation}
The two expressions agree at $i=r$. As usual, an empty range is omitted.

For the odd character $\lambda^{2r+1}$, the residues are
\begin{equation}
\label{eq:odd-arm-residues}
t_{2r+1}=\frac12,
\qquad
R_i^{(2r+1)}=
\begin{cases}
\displaystyle
\frac{(k-r)i}{k}+\frac{k-i}{2k},
&1\le i\le r,\\[0.8em]
\displaystyle
\frac{(2r+1)(k-i)}{2k},
&r\le i\le k-1.
\end{cases}
\end{equation}
Again the two expressions agree at $i=r$.

We use the sign convention of equation~\eqref{eq:master-err-intro}: a local correction coefficient is the negative of the corresponding contribution to the global Euler characteristic. The arm contribution includes
\begin{itemize}
    \item the self-intersections of $F_1,\ldots,F_{k-1}$,
    \item the intersections of consecutive components of the chain, and
    \item the intersection $F_1\cdot T=1$.
\end{itemize}
It does not include the self-intersection of $T$.

\begin{proposition}[Cyclic arm correction]
\label{prop:cyclic-arm-correction}
With the notation above, define
\begin{equation}
\label{eq:cyclic-arm-class}
\begin{aligned}
    \alpha_k
    :={}&
    \sum_{r=0}^{k-1}
      \frac{r(k-r)}{2k}\,\lambda^{2r}\\
    &+
    \sum_{r=0}^{k-1}
      \left(
        \frac{r(k-r)}{2k}
        +\frac{1-k}{8k}
      \right)\lambda^{2r+1}
    \in
    \operatorname{Rep}(C_k)\otimes_{\mathbb Z}\mathbb Q.
\end{aligned}
\end{equation}
Then $\alpha_k$ is the correction contributed by the $A_{k-1}$ arm and its intersection with $T$, with the self-intersection of $T$ omitted.
\end{proposition}

\begin{proof}
Fix a character $\lambda^m$. Since all these curves are $(-2)$-curves, their canonical intersections vanish. The terms assigned to the arm in the global residue formula give
\[
    m_{q,\lambda^m}
    =\sum_{i=1}^{k-1}\bigl(R_i^{(m)}\bigr)^2
     -\sum_{i=1}^{k-2}R_i^{(m)}R_{i+1}^{(m)}
     -t_mR_1^{(m)}.
\]

Put $R_0^{(m)}=t_m$ and $R_k^{(m)}=0$. Formulas \eqref{eq:even-arm-residues} and \eqref{eq:odd-arm-residues} extend to $i=0$ and $i=k$ with these values. Rearranging the sum gives
\[
    m_{q,\lambda^m}
    =\frac12\sum_{i=1}^{k-1}R_i^{(m)}
       \bigl(2R_i^{(m)}-R_{i-1}^{(m)}-R_{i+1}^{(m)}\bigr)
     -\frac12t_mR_1^{(m)}.
\]
Write $m=2r$ or $m=2r+1$. The function $i\mapsto R_i^{(m)}$ is linear on $0\leq i\leq r$ and on $r\leq i\leq k$, and its slope drops by $1$ at $i=r$. Hence the bracket equals $1$ at $i=r$ and vanishes for the other $1\leq i\leq k-1$, so
\[
    m_{q,\lambda^m}=\frac12R_r^{(m)}-\frac12t_mR_1^{(m)},
\]
where the first term is absent when $r=0$.

For $m=2r$, one has $t_m=0$, and therefore $m_{q,\lambda^{2r}}=r(k-r)/(2k)$.

For $m=2r+1$ with $r\geq1$, one has $t_m=1/2$,
$R_r^{(m)}=r(k-r)/k+(k-r)/(2k)$, and
$R_1^{(m)}=(k-r)/k+(k-1)/(2k)$. Hence
\[
    m_{q,\lambda^{2r+1}}
    =\frac{r(k-r)}{2k}
     +\frac{1-k}{8k}.
\]
For $r=0$, one has $m_{q,\lambda}=-\frac14R_1^{(1)}=(1-k)/(8k)$, which is the same formula.
Summing these coefficients over the characters of $C_k$ proves the formula.
\end{proof}

\begin{remark}[Including the central self-intersection]
\label{rem:full-cyclic-correction}
Let $\varepsilon$ be the nontrivial character of $H=\{\pm I\}\cong\mathbb Z_2$. If the self-intersection $T^2=-2$ is also included, the cyclic correction becomes
\begin{equation}
\label{eq:full-cyclic-correction}
    \beta_k
    =\alpha_k
     +\frac14\operatorname{Ind}_{H}^{C_k}(\varepsilon).
\end{equation}
Indeed, $\operatorname{Ind}_{H}^{C_k}(\varepsilon)=\sum_{r=0}^{k-1}\lambda^{2r+1}$, so $\beta_k$ has the same even coefficients as $\alpha_k$, and its coefficient of $\lambda^{2r+1}$ is
\begin{equation}
\label{eq:full-cyclic-character-form}
    \frac{r(k-r)}{2k}+\frac{k+1}{8k}.
\end{equation}
When several arms meet the same curve $T$, its self-intersection must be counted only once.
\end{remark}

\subsection{Assembly of the ADE correction}

The self-intersection of the central curve contributes
\begin{equation}
\label{eq:central-reflection-correction}
    \sigma_p
    :=\frac14
      \operatorname{Ind}_{H}^{G_p}(\varepsilon)
    \in
    \operatorname{Rep}(G_p)\otimes_{\mathbb Z}\mathbb Q.
\end{equation}
To see this, note that the residue along $T$ is $0$ on the trivial character of $H$ and $1/2$ on $\varepsilon$. Since $T^2=-2$ and $K_S\cdot T=0$, the term $-\frac12\bigl(t^2T^2+tK_S\cdot T\bigr)$ equals $\frac14$ on $\varepsilon$ and $0$ on the trivial character. Frobenius reciprocity, as in the proof of Theorem~\ref{thm:blowup-assembly}, gives $\sigma_p$.

Let $\Omega_1,\Omega_2,\Omega_3\subset E$ be the three special orbits, and choose representatives
\[
    q_s\in\Omega_s,
    \qquad
    C_s:=G_{q_s}\cong\mathbb Z_{2k_s},
    \qquad s=1,2,3.
\] The action on the normal coordinate function determines a character $\lambda_s\colon C_s\to\mathbb C^*$, given by $g\cdot u=\lambda_s(g)u$. The pair $(C_s,\lambda_s)$ is part of the local data. Two different special orbits may have the same underlying isotropy subgroup but different normal characters.

Let $\alpha_{k_s,q_s}\in\operatorname{Rep}(C_s)\otimes_{\mathbb Z}\mathbb Q$ be the class of Proposition~\ref{prop:cyclic-arm-correction}, formed using $\lambda_s$.

By the local correction class function at $p$ we mean the class of the terms of the global residue formula on the resolution of $X'/G_p$ that involve the exceptional configuration over $p$. Since the action is free off $p$ near $p$, this is the full contribution of the orbit of $p$. By Lemma~\ref{lem:general-hj-model-independence}, it does not depend on the admissible model used to compute it. In Section~\ref{sec:applications} it is compared with the holomorphic Lefschetz class.

\begin{theorem}[ADE assembly formula]
\label{thm:blowup-assembly}
The local correction class function at $p$ is
\begin{equation}
\label{eq:rho-three-orbit-sum}
    \rho_p
    =\sigma_p
     +\sum_{s=1}^3
       \operatorname{Ind}_{C_s}^{G_p}
       \bigl(\alpha_{k_s,q_s}\bigr)
    \in
    \operatorname{Rep}(G_p)\otimes_{\mathbb Z}\mathbb Q.
\end{equation}
For every $\xi\in\operatorname{Irr}(G)$, the contribution of the orbit of $p$ is
\begin{equation}
\label{eq:multiplicity-pairing}
    m_{p,\xi}
    =\left\langle
       \rho_p,
       \operatorname{Res}_{G_p}^{G}\xi
     \right\rangle_{G_p}.
\end{equation}
The class $\rho_p$ is independent of the chosen representatives $q_s$.
\end{theorem}

\begin{proof}
By equivariant birational invariance, blowing up the orbit of $p$ does not change $\chi_G(X,\mathcal O_X)$. On the resolution of $X'/G_p$, the exceptional divisor consists of the central curve $T$ and three cyclic chains. Distinct chains do not meet, and each chain meets only $T$. The residue expression therefore splits into the self-intersection contribution of $T$ and the three arm contributions.

The central term is $\sigma_p$. For a fixed orbit $\Omega_s$, the contribution from all of its points is obtained from the calculation at one representative by induction from $C_s$ to $G_p$. Its class is therefore
\[
    \operatorname{Ind}_{C_s}^{G_p}
    \bigl(\alpha_{k_s,q_s}\bigr).
\]
Adding the central term and the three induced arm terms gives \eqref{eq:rho-three-orbit-sum}.

If another representative of $\Omega_s$ is chosen, its isotropy subgroup and normal character are conjugate to those at $q_s$. Induction from conjugate data gives the same class function on $G_p$. Hence $\rho_p$ is independent of the representatives.

To check the character multiplicities, use the rank-$\xi(1)$ multiplicity
sheaf $\mathcal F_\xi$ of Section~\ref{sec:preliminaries} for the global
quotient. Near the image of $q_s$, its natural extension splits into the
local line bundles. The summand for $\lambda_s^m$ occurs with multiplicity
$\langle\operatorname{Res}_{C_s}^G\xi,\lambda_s^m\rangle_{C_s}$.
Frobenius reciprocity gives
\[
\begin{aligned}
    \left\langle
      \operatorname{Ind}_{C_s}^{G_p}
      \alpha_{k_s,q_s},
      \operatorname{Res}_{G_p}^{G}\xi
    \right\rangle_{G_p}
    &=
    \left\langle
      \alpha_{k_s,q_s},
      \operatorname{Res}_{C_s}^{G}\xi
    \right\rangle_{C_s}.
\end{aligned}
\]
This is the contribution of the orbit $\Omega_s$. There is no additional factor of $\xi(1)$, since we use the multiplicity sheaf. Summing the three orbit terms and the central term proves \eqref{eq:multiplicity-pairing}.
\end{proof}

\begin{remark}[Using the full cyclic classes]
If one uses the classes $\beta_{k_s}$ from \eqref{eq:full-cyclic-correction}, formed with the normal character $\lambda_s$ at each orbit, then transitivity of induction gives
\begin{equation}
\label{eq:assembly-full-cyclic}
    \rho_p
    =\sum_{s=1}^3
      \operatorname{Ind}_{C_s}^{G_p}(\beta_{k_s})
      -2\sigma_p.
\end{equation}
Each induced class contains one copy of the self-intersection contribution of the same central curve $T$. The sum counts that term three times, so two copies must be removed.
\end{remark}

\subsection{The ADE formulas}

The assembly theorem gives the four families uniformly. Each term uses
the pair $(C_s,\lambda_s)$ from its special orbit. We write $\alpha_{k_s}$
for $\alpha_{k_s,q_s}$ below, always using the corresponding character
$\lambda_s$. In the order of Table~\ref{tab:exceptional-orbits}, the
isotropy orders $2k_s$ give
\[
\renewcommand{\arraystretch}{1.15}
\begin{array}{c|cccc}
    \text{type} & D_{n+2} & E_6 & E_7 & E_8\\ \hline
    (k_1,k_2,k_3) & (n,2,2) & (3,2,3) & (4,2,3) & (5,2,3)
\end{array}
\]
and in each case
\begin{equation}
\label{eq:ADE-correction-class}
    \rho_p
    =\sigma_p
     +\operatorname{Ind}_{C_1}^{G_p}(\alpha_{k_1})
     +\operatorname{Ind}_{C_2}^{G_p}(\alpha_{k_2})
     +\operatorname{Ind}_{C_3}^{G_p}(\alpha_{k_3}).
\end{equation}
For instance, for $D_{n+2}$ with $G_p=\operatorname{Dic}_n$ and $n\geq2$, one has $C_1\cong\mathbb Z_{2n}$ and $C_2\cong C_3\cong\mathbb Z_4$, so the correction is built from $\alpha_n$ and two copies of $\alpha_2$. For $E_6$, the two arms of length two can be represented by the two eigenlines of one element of order six. They have the same isotropy subgroup but inverse normal characters.

These formulas reduce every non-cyclic ADE correction to explicit class functions on cyclic isotropy subgroups. No logarithmic modification of a non-abelian isotypic bundle is required.

%% file: ch-general-hj.tex
\ProvidesFile{ch-general-hj.tex}

\section{The General Hirzebruch--Jung Calculation}
\label{sec:general-hj}

Blowing up a fixed point replaces its isotropy subgroup by diagonal isotropy subgroups along
the exceptional curve. After removing the quasi-reflections, each local quotient
is cyclic. We now compute the residues for these diagonal actions and explain
how they enter the global Riemann--Roch formula.

The local action determines the exceptional chain, its residues, and its
intersections with the boundary curves. Boundary self-intersections and
canonical intersections come from the surrounding surface. We therefore
compute local contributions to a global Riemann--Roch formula, rather than
Euler characteristics on the local quotient germs.

\subsection{The diagonal action and its reflection subgroups}

Let $p\in X$ have isotropy subgroup $G_p$, and let
$f\colon (X',E)\to(X,p)$ be the blow-up at $p$, so that $E\cong\mathbb P^1$
and $E^2=-1$. Globally, we blow up the full orbit of $p$. Fix $q\in E$ with
isotropy subgroup $G_q$. By Lemma~\ref{lem:diag-formal}, the action of $G_q$
is diagonal in local coordinates $(x,y)$ with $E=\{x=0\}$. These are the
coordinates $(u,v)$ of Section~\ref{sec:blowup}. The local calculation also
applies directly to any faithful diagonal surface action, without a
preliminary blow-up.

We keep $q$ fixed and omit it from the characters and numerical data.
Characters label the action on coordinate functions,
$(g\cdot h)(z)=h(g^{-1}z)$, so they are the inverses of the corresponding
tangent characters. Write
\[
    g\cdot x=\alpha(g)x,
    \qquad
    g\cdot y=\beta(g)y,
\]
so that $x^u y^v$ has character $\alpha^u\beta^v$.

The faithful diagonal representation identifies $G_q$ with a finite
subgroup of $(\mathbb C^*)^2$, so $G_q$ is abelian. Its two reflection
subgroups $H_x:=\ker\beta$ and $H_y:=\ker\alpha$ fix $\{x=0\}$ and
$\{y=0\}$ pointwise, respectively. Set $r_x:=|H_x|$, $r_y:=|H_y|$, and
$H_q:=H_xH_y$. Effectiveness gives $H_x\cap H_y=\{1\}$, so
$H_q\cong H_x\times H_y$. This is the subgroup generated by all
quasi-reflections. The quotient by $H_q$ is smooth, as also follows from the
Chevalley--Shephard--Todd theorem
\cite{shephard_finite_1954,chevalley_invariants_1955}, with coordinates
\[
    X=x^{r_x},\qquad Y=y^{r_y}.
\]

The residual group $C_q:=G_q/H_q$ acts faithfully on each coordinate.
Indeed, $\alpha|_{H_x}$ is faithful of order $r_x$. If
$\alpha(g)^{r_x}=1$, choose $h\in H_x$ with $\alpha(h)=\alpha(g)$.
Then $gh^{-1}\in H_y$, so $g\in H_q$.
The same argument applies to $Y$. Hence $C_q$ is cyclic and has no
quasi-reflections. Write $n:=|C_q|=|G_q|/(r_xr_y)$.

If $n>1$, choose its generator $\bar g$ so that
\begin{equation}
\label{eq:general-hj-residual-action}
    \bar g\cdot X=\zeta_nX,
    \qquad
    \bar g\cdot Y=\zeta_n^aY,
    \qquad
    1\leq a<n,\quad \gcd(a,n)=1,
\end{equation}
where $\zeta_n=\exp(2\pi i/n)$. Thus $X$ remains the coordinate
vanishing on the image of $E$, and the quotient has type
$\frac1n(1,a)$.
If $n=1$, the quotient is smooth and there is no exceptional chain.

\subsection{The resolution and its orientation}

We use the standard toric description of Hirzebruch--Jung resolutions,
see \cite[\S\S~2--3]{reid_surface_nodate} and
\cite{cox_toric_2011}. We specify the orientation because it determines
which continued fraction occurs.

Suppose first that $n>1$, and let $a^{\vee}\in\{1,\ldots,n-1\}$ be the
inverse of $a$ modulo $n$. We order the chain from the boundary $X=0$ to the
boundary $Y=0$. With this order, write
\begin{equation}
\label{eq:general-hj-continued-fraction}
    \frac{n}{a^{\vee}}=[b_1,\ldots,b_\ell]
    :=b_1-\cfrac{1}{b_2-\cfrac{1}{\ddots-\cfrac{1}{b_\ell}}},
    \qquad b_i\geq2.
\end{equation}
The coefficients are obtained by the ceiling Euclidean algorithm: start
with $(u_0,u_1)=(n,a^{\vee})$, and set
$b_i=\lceil u_{i-1}/u_i\rceil$ and $u_{i+1}=b_i u_i-u_{i-1}$
until $u_{\ell+1}=0$. Ordering the chain from $Y=0$ instead gives
$n/a=[b_\ell,\ldots,b_1]$.

The quotient is toric, with lattice
$N=\mathbb Z^2+\frac1n(1,a)\mathbb Z$.
The primitive rays of its minimal resolution, in the stated order, are
\[
    v_0=(1,0),\qquad
    v_1=\left(\frac{a^{\vee}}n,\frac1n\right),\qquad\ldots,\qquad
    v_\ell=\left(\frac1n,\frac an\right),\qquad
    v_{\ell+1}=(0,1),
\]
with the intermediate rays determined by
\begin{equation}
\label{eq:general-hj-ray-recurrence}
    v_{i+1}=b_i v_i-v_{i-1}
    \qquad (1\leq i\leq\ell).
\end{equation}
Consecutive rays form a basis of $N$.

Let $Y_q$ be the quotient germ at the image of $q$, and let
$\mu_q\colon\widetilde Y_q\to Y_q$ be its minimal resolution. Denote the
exceptional curves by $F_1,\ldots,F_\ell$, and the strict transforms of the
coordinate boundaries by $B_x,B_y$. Then
\[
    F_i^2=-b_i,\qquad F_i\cdot F_{i+1}=1,
    \qquad B_x\cdot F_1=B_y\cdot F_\ell=1,
\]
and all other intersections between distinct members of this list vanish.
Put
\[
    J:=(F_i\cdot F_j)_{1\leq i,j\leq\ell},
    \qquad
    \boldsymbol\kappa:=(b_1-2,\ldots,b_\ell-2)^T.
\]
By adjunction \cite[Chapter~V, \S~1]{hartshorne_algebraic_1977}, a
canonical divisor $K$ on $\widetilde Y_q$ satisfies
\begin{equation}
\label{eq:general-hj-canonical-intersections}
    K\cdot F_i=b_i-2.
\end{equation}
These are degrees on the proper curves $F_i$. Boundary
self-intersections are not part of the local data.

When $n=1$, set $\ell=0$, take $v_0=(1,0)$ and $v_1=(0,1)$, and
omit $J$ and $\boldsymbol\kappa$. In this case $B_x$ and $B_y$
meet transversely at the smooth quotient point.

\subsection{The character modules}

For $\chi\in\operatorname{Irr}(G_q)$, let
$\mathcal M_\chi$ be the corresponding rank-one summand of the finite
direct image. In the linear local model it is generated by the monomials
\[
    x^u y^v,\qquad
    (u,v)\in S_\chi:=
    \{(u,v)\in\mathbb Z_{\geq0}^2:\alpha^u\beta^v=\chi\}.
\]
Since the coordinates of Lemma~\ref{lem:diag-formal} linearize the action,
the same description applies to the analytic germ.

Since $\alpha|_{H_x}$ and $\beta|_{H_y}$ generate the character groups of
the cyclic groups $H_x$ and $H_y$, there are unique integers
$0\leq d_x<r_x$ and $0\leq d_y<r_y$, depending on $\chi$, such that
\[
    \chi|_{H_x}=\alpha^{d_x}|_{H_x},
    \qquad
    \chi|_{H_y}=\beta^{d_y}|_{H_y}.
\]
If a reflection subgroup is trivial, the corresponding integer is zero.
The character $\chi\alpha^{-d_x}\beta^{-d_y}$
is trivial on $H_q$ and therefore descends to $C_q$.
For $n>1$, define $c_\chi\in\{0,\ldots,n-1\}$ by
\[
    (\chi\alpha^{-d_x}\beta^{-d_y})(\bar g)=\zeta_n^{c_\chi}.
\]
For $n=1$, put $c_\chi=0$.

When $n>1$, every monomial of character $\chi$ has the form
\[
    x^{d_x}y^{d_y}X^iY^j,
    \qquad i,j\geq0,
    \qquad i+aj\equiv c_\chi\pmod n.
\]
Let $[m]_n$ denote the representative of $m$ in
$\{0,\ldots,n-1\}$. Then $\mathcal M_\chi$ is generated over the
invariant ring by the finite list
\begin{equation}
\label{eq:general-hj-monomial-generators}
    s_{\chi,j}:=x^{d_x}y^{d_y}X^{[c_\chi-aj]_n}Y^j,
    \qquad 0\leq j<n.
\end{equation}
Indeed, $X^n$ and $Y^n$ are invariant, so the exponents can first be
reduced modulo $n$. The congruence then fixes the exponent of $X$.
For $n=1$, the module is free with generator $x^{d_x}y^{d_y}$.

\subsection{The natural extension and its residues}

The natural extension
$\mathcal L^{\mathrm{nat}}_\chi:=\mu_q^*\mathcal M_\chi/\operatorname{torsion}$
is a line bundle by
Proposition~\ref{extension_prop}, see also
\cite[Proposition~2.7]{gustavsen_deformations_2018}.

The lattice $N$ uses the coordinates $X,Y$, not $x,y$. Thus the
valuation of an original monomial is given by the toric
monomial--ray pairing \cite{cox_toric_2011}:
\begin{equation}
\label{eq:general-hj-monomial-valuation}
    \nu_i(x^u y^v)
    =\left\langle
        \left(\frac{u}{r_x},\frac{v}{r_y}\right),v_i
      \right\rangle.
\end{equation}
For a semi-invariant monomial this is a rational valuation: an invariant
power has an integral valuation, and we divide by the exponent.

\begin{proposition}[Residues of the natural extension]
\label{prop:general-hj-natural-residues}
The connection induced by $d$ on $\mathcal L^{\mathrm{nat}}_\chi$
is logarithmic along the exceptional chain and the two boundary curves.
Its boundary residues are
\[
    t_{x,\chi}=\frac{d_x}{r_x},\qquad
    t_{y,\chi}=\frac{d_y}{r_y}.
\]
For $1\leq i\leq\ell$, its exceptional residue is
\begin{equation}
\label{eq:general-hj-minimal-valuation}
    R_{i,\chi}
    =\min_{(u,v)\in S_\chi}
      \left\langle
        \left(\frac{u}{r_x},\frac{v}{r_y}\right),v_i
      \right\rangle.
\end{equation}
More explicitly, if $v_i=(v_i^{(1)},v_i^{(2)})$, then
\begin{equation}
\label{eq:general-hj-finite-residue-formula}
    R_{i,\chi}
    =t_{x,\chi}v_i^{(1)}
      +t_{y,\chi}v_i^{(2)}
     +\min_{0\leq j<n}
       \left([c_\chi-aj]_n v_i^{(1)}+jv_i^{(2)}\right).
\end{equation}
There are no exceptional residues when $n=1$.
\end{proposition}

\begin{proof}
Work on one smooth toric chart of $\widetilde Y_q$, with coordinates
$z_1,z_2$ and torus-fixed point $o$. There the natural extension is
generated by the images of the monomials
\eqref{eq:general-hj-monomial-generators}. Their images span its
one-dimensional fiber at $o$, so by Nakayama's lemma one of them, say $s$,
generates the stalk at $o$. Every other generator is $s$ times an invariant
Laurent monomial, that is, a Laurent monomial in $z_1,z_2$. Since this
monomial is regular at $o$, both of its exponents are nonnegative. Hence $s$
is a frame on the whole chart, and it attains the minimum valuation along
both adjacent rays simultaneously. This is where the local freeness of the
natural extension is used.

An invariant power of $s$ is a Laurent monomial in $z_1,z_2$. Hence
\[
    \nabla s
    =s\otimes\left(c_1\frac{dz_1}{z_1}
                         +c_2\frac{dz_2}{z_2}\right)
\]
for rational constants $c_1,c_2$. The connection is therefore logarithmic.
The two constants are the valuations of $s$ along the chart boundaries.
Since $s$ generates the natural extension, these valuations are the minima
among all its monomial generators. This proves
\eqref{eq:general-hj-minimal-valuation}, and the finite generating list gives
\eqref{eq:general-hj-finite-residue-formula}.

Along $B_x$ the minimum exponent of $x$ is $d_x$: for $n>1$, a choice
of $j$ makes $[c_\chi-aj]_n=0$. Along $B_y$ the minimum exponent of
$y$ is $d_y$, obtained with $j=0$. Equation
\eqref{eq:general-hj-monomial-valuation} gives the stated boundary residues.
The case $n=1$ follows directly from the frame $x^{d_x}y^{d_y}$.
\end{proof}

\subsection{From residues to correction terms}

Let $\mu\colon S\to Y'$ be the minimal resolution of the global quotient
after the blow-ups. Near the image of $q$, it is the minimal resolution
$\mu_q$ described above. The logarithmic divisor $\Delta=\sum_a\Delta_a$
consists of the exceptional curves of $\mu$ and the strict transforms of
the branch curves. In the setting of Theorem~\ref{thm:intro-global}, it is
SNC, see Section~\ref{subsec:general-hj-global-proof}.

For a rank-$r$ bundle $V$ with a flat logarithmic connection, let $R_a$ be
its residue along $\Delta_a$. Equation~\eqref{eq:prelim-residue-euler},
based on \cite{ohtsuki_residue_1982,esnault_logarithmic_1986}, gives
\begin{equation}
\label{eq:general-hj-global-residue-correction}
\begin{aligned}
    r\chi(\mathcal O_S)-\chi(S,V)
    &=-\frac12\sum_a
       \left(\operatorname{tr}(R_a^2)\Delta_a^2
             +\operatorname{tr}(R_a)K_S\cdot\Delta_a\right)\\
    &\quad-\sum_{a<b}
        \sum_{z\in\Delta_a\cap\Delta_b}
           \operatorname{tr}\bigl(R_a(z)R_b(z)\bigr).
\end{aligned}
\end{equation}
Each intersection point is counted once. The sign is the correction
convention of \eqref{eq:master-err-intro}. A local character module need not
extend to a global line bundle: its residues enter this formula through
the global multiplicity bundles of Section~\ref{sec:preliminaries}.

We now specify which terms are assigned to a local quotient point.
For $n>1$, include the self-intersections and canonical intersections of
its exceptional chain, the intersections inside the chain, and its
intersections with both boundary curves. Exclude the self-intersections and
canonical intersections of the boundary curves. For $n=1$, include only
the intersection $B_x\cap B_y$.

\begin{proposition}[Local contribution to the global residue formula]
\label{prop:general-hj-residue-formula}
For $\chi\in\operatorname{Irr}(G_q)$, let
$\mathbf R_\chi=(R_{1,\chi},\ldots,R_{\ell,\chi})^T$ be the
exceptional residue vector of the natural extension, and let
$t_{x,\chi},t_{y,\chi}$ be its boundary residues. With the allocation above,
the correction coefficient is
\begin{equation}
\label{eq:general-hj-universal-correction}
    m_\chi=
    \begin{cases}
    -\dfrac12\left(
       \mathbf R_\chi^T J\mathbf R_\chi
       +\boldsymbol\kappa^T\mathbf R_\chi\right)
       -t_{x,\chi}R_{1,\chi}
       -t_{y,\chi}R_{\ell,\chi}, & n>1,\\[8pt]
    -t_{x,\chi}t_{y,\chi}, & n=1.
    \end{cases}
\end{equation}
\end{proposition}

\begin{proof}
For $n>1$, the chain self-intersections and internal intersections give
$-\frac12\mathbf R_\chi^T J\mathbf R_\chi$.
Adjunction gives the canonical term
$-\frac12\boldsymbol\kappa^T\mathbf R_\chi$, and the two
boundary intersections contribute
$-t_{x,\chi}R_{1,\chi}-t_{y,\chi}R_{\ell,\chi}$.
For $n=1$, only the transverse intersection of the boundaries remains.
\end{proof}

In particular, a smooth residual quotient can still contribute when both
reflection subgroups are nontrivial. The local formula never requires the
self-intersection of either boundary curve.

\begin{remark}[Other logarithmic extensions]
\label{rem:general-hj-other-extensions}
The natural extension suffices here, but the comparison with
Section~\ref{sec:preliminaries} can also be made explicit. Any other
line-bundle extension of the same flat bundle that agrees with the natural
extension and its connection away from the exceptional chain, including
along the boundaries, is $\mathcal L^{\log}_\chi=\mathcal L^{\mathrm{nat}}_\chi(D)$
with $D=\sum_{i=1}^{\ell}d_iF_i$ and $d_i\in\mathbb Z$. Its boundary residues
are unchanged, and $R_{i,\chi}^{\log}=R_{i,\chi}-d_i$. The class
$[\mathcal L^{\log}_\chi]-[\mathcal L^{\mathrm{nat}}_\chi]$ is supported on
the proper curve $F_1\cup\cdots\cup F_\ell$, and Riemann--Roch on the
exceptional curves \cite[Chapter~IV, \S~1]{hartshorne_algebraic_1977} gives
its Euler characteristic
\begin{equation}
\label{eq:general-hj-comparison-number}
    \delta_\chi
       =c_1(\mathcal L^{\mathrm{nat}}_\chi)\cdot D
          +\frac12(D^2-K\cdot D).
\end{equation}
Here $c_1(\mathcal L^{\mathrm{nat}}_\chi)\cdot F_i
=b_iR_{i,\chi}-R_{i-1,\chi}-R_{i+1,\chi}$, with $R_{0,\chi}=t_{x,\chi}$ and
$R_{\ell+1,\chi}=t_{y,\chi}$, by the monomial transition functions.
Substituting $R_{i,\chi}^{\log}$ in
\eqref{eq:general-hj-universal-correction} changes its value by
$-\delta_\chi$. Thus the correction for the natural extension is obtained by
adding $\delta_\chi$ to the expression computed with the new residues, as in
\eqref{eq:prelim-extension-comparison}.
\end{remark}

\subsection{The central curve and assembly}

We now re-use the subscript $q$ on local data when comparing different
points of $E$.
Let $H_E\subset G_p$ be the subgroup acting trivially on $E$.
It consists of the scalar matrices in the tangent representation at $p$,
so it is cyclic and central. Write $r_E:=|H_E|$.
It is the generic isotropy subgroup along $E$, and $H_{x,q}=H_E$ in every
diagonal chart: fixing an open subset of $E$ pointwise is equivalent
to acting trivially on $E$. Hence $r_{x,q}=r_E$.
Let $\overline E$ be the image of $E$ on the quotient before resolving its
singularities.

For the local quotient map $\pi$, one has $\pi^*\overline E=r_EE$. The
projection formula and $E^2=-1$ give
\begin{equation}
\label{eq:general-hj-central-square-singular}
    \overline E^{\,2}=-\frac{r_E^2}{|G_p|},
\end{equation}
a rational intersection number on the normal quotient.

Let $T$ be the strict transform of $\overline E$ on $S$. At each
singularity of the quotient lying on $\overline E$, index the exceptional
chain from $T$, as above. The numerical pullback of Mumford
\cite{mumford_topology_1961} is
\[
    \mu^*\overline E=T+\sum_{q:\,n_q>1}\sum_{i=1}^{\ell_q}
                   v_{q,i}^{(1)}F_{q,i}.
\]
Indeed, the right-hand side has zero intersection with every exceptional
curve: with $v_{q,0}^{(1)}=1$ the coefficient of $T$ and
$v_{q,\ell_q+1}^{(1)}=0$, the ray recurrence gives
\[
    v_{q,i-1}^{(1)}-b_{q,i}v_{q,i}^{(1)}+v_{q,i+1}^{(1)}=0
    \qquad(1\leq i\leq\ell_q).
\]
Intersecting with $T$ gives
\begin{equation}
\label{eq:general-hj-central-square-resolved}
    T^2=-\frac{r_E^2}{|G_p|}
          -\sum_{q:\,n_q>1}\frac{a_q^{\vee}}{n_q}.
\end{equation}
The sum has one representative for each $G_p$-orbit whose image on
$\overline E$ is a singular point of the quotient. Smooth quotient points
contribute nothing. The inverse $a_q^{\vee}$ occurs because the chain
starts at $X=0$. The same computation gives the self-intersections of the
branch curves on $S$. If a branch curve passes through a singular quotient
point as the image of $\{x=0\}$, the decrease is $a_q^\vee/n_q$, as above.
If it passes as the image of $\{y=0\}$, the coefficients of its pullback
along the chain are the $v_{q,i}^{(2)}$, and the decrease is
$v_{q,\ell_q}^{(2)}=a_q/n_q$.

Since $T\cong E/G_p\cong\mathbb P^1$, adjunction gives
$K_S\cdot T=-2-T^2$. Thus both numerical terms involving only $T$ are
determined.

Let $\Omega_1,\ldots,\Omega_t\subset E$ be the orbits with
isotropy subgroup larger than $H_E$, and choose $q_s\in\Omega_s$.
There are finitely many such orbits, and there are none if the projective
action is trivial. Each quotient singularity on $\overline E$, and each
intersection of $\overline E$ with another branch curve, comes from one of
these orbits. Put
\[
    \rho_{q_s}^{\mathrm{loc}}
       :=\sum_{\chi\in\operatorname{Irr}(G_{q_s})}m_{q_s,\chi}\chi,
\]
which includes the smooth-intersection term when $n_{q_s}=1$.

To define the central term, let $\lambda_E$ be the character of $H_E$
on a normal coordinate function along $E$. It is faithful. For
$\psi=\lambda_E^d$, $0\leq d<r_E$, put $t_\psi=d/r_E$ and set
\[
    \sigma_p
    :=\operatorname{Ind}_{H_E}^{G_p}
      \left(
       \sum_{\psi\in\operatorname{Irr}(H_E)}
        -\frac12\left(t_\psi^2T^2+t_\psi K_S\cdot T\right)\psi
      \right).
\]
If $H_E$ is trivial, this class is zero.

Let $\mathcal E_p$ be the union of $T$ and all the chains over
$\overline E$. Assign to $\mathcal E_p$ every term of
\eqref{eq:general-hj-global-residue-correction} involving one of its
components, including intersections with branch curves outside
$\mathcal E_p$. The resulting class on $G_p$ is
\begin{equation}
\label{eq:general-hj-global-assembly}
    \rho_p^{\mathrm{exc}}
       =\sigma_p
        +\sum_{s=1}^t
           \operatorname{Ind}_{G_{q_s}}^{G_p}
                         \rho_{q_s}^{\mathrm{loc}}.
\end{equation}
The terms involving only $T$ occur once.

To verify the character multiplicities, use the multiplicity sheaf
$\mathcal F_\eta:=\operatorname{Hom}_{G_p}(W_\eta,\pi_*\mathcal O_{X'})$
of rank $\eta(1)$ in the $G_p$-quotient model. Near the image of $q_s$,
its natural extension splits into the local line bundles, with
$\mathcal L^{\mathrm{nat}}_\chi$ occurring
$\langle\operatorname{Res}_{G_{q_s}}^{G_p}\eta,\chi\rangle_{G_{q_s}}$
times. Frobenius reciprocity \cite{serre_linear_1977} gives
\[
    \sum_\chi
       \left\langle\operatorname{Res}_{G_{q_s}}^{G_p}\eta,
              \chi\right\rangle_{G_{q_s}}m_{q_s,\chi}
    =\left\langle
       \operatorname{Ind}_{G_{q_s}}^{G_p}\rho_{q_s}^{\mathrm{loc}},
       \eta\right\rangle_{G_p}.
\]
The same argument at the generic point of $T$ gives $\sigma_p$.
There is no additional factor of $\eta(1)$: we use the multiplicity
sheaf, rather than the full isotypic summand. For a character $\xi$ of the
ambient group $G$, the corresponding coefficient is
$\langle\rho_p^{\mathrm{exc}},\operatorname{Res}_{G_p}^G\xi\rangle_{G_p}$.

By Proposition~\ref{prop:general-hj-natural-residues}, the natural
connections are logarithmic near $\mathcal E_p$, so no comparison term is
needed there. Once every isotropy subgroup acts diagonally, as in
Theorem~\ref{thm:intro-global}, the natural extension is logarithmic
everywhere. If a different extension is used at other points, it can be
glued to the natural one along the smooth branch locus, where both agree
with Deligne's extension with residues in $[0,1)$
\cite{deligne_equations_1970}, and the comparison is then made as in
Section~\ref{sec:preliminaries}.

If the original action is free off $p$ in a neighborhood of $p$, no
other branch curve meets this configuration, and
$\rho_p^{\mathrm{exc}}$ is the full local correction $\rho_p$.
If fixed curves pass through $p$, their strict transforms also have
self-intersection and canonical-intersection terms. Formula
\eqref{eq:general-hj-global-assembly} gives only the exceptional part.
The remaining branch terms stay in the global formula.

In summary, a term of \eqref{eq:general-hj-global-residue-correction} is
local when it involves an exceptional curve of $\mu$ or an intersection
point of two components of $\Delta$. It is assigned to the quotient point
over which it lies. The self-intersection and canonical terms of $T$ are
assigned to $p$ through $\sigma_p$, since
\eqref{eq:general-hj-central-square-resolved} computes them from local data.
The self-intersection and canonical terms of every other branch curve remain
in the global formula.

\subsection{Independence of the local model}
\label{subsec:general-hj-model-independence}

Suppose that $G_p$ acts freely on $U\setminus\{p\}$ for some $G_p$-stable
neighborhood $U$ of $p$, and let $\bar p$ be the image of $p$ in $U/G_p$.
An \emph{admissible model} $M$ at $p$ consists of a $G_p$-equivariant
composition $U'\to U$ of blow-ups of points over $p$, possibly empty, and a
resolution $\nu\colon S_U\to U'/G_p$ that is an isomorphism over the smooth
locus, such that the reduced preimage $\Delta_{\bar p}$ of $\bar p$ in $S_U$
is SNC and the natural extensions of the multiplicity sheaves $\mathcal F_\eta$,
$\eta\in\operatorname{Irr}(G_p)$, are logarithmic along $\Delta_{\bar p}$.
The terms of \eqref{eq:general-hj-global-residue-correction} that involve
components of $\Delta_{\bar p}$ define a class
$\rho_p^{M}\in\operatorname{Rep}(G_p)\otimes_{\mathbb Z}\mathbb Q$: its
pairing with $\eta$ is the sum of these terms for the natural extension of
$\mathcal F_\eta$. Since the action is free off $p$, no other component of
the logarithmic divisor meets $\Delta_{\bar p}$, so these terms are
determined by $M$. By Proposition~\ref{prop:general-hj-natural-residues},
the minimal resolution of a cyclic quotient without quasi-reflections is
admissible. So is the blow-up of $p$ followed by the minimal resolution, and
then $\rho_p^M=\rho_p^{\mathrm{exc}}$.

\begin{lemma}[Independence of the local model]
\label{lem:general-hj-model-independence}
If $G_p$ acts freely on a punctured neighborhood of $p$, then $\rho_p^M$ is
the same for all admissible models $M$ at $p$. It depends only on the
conjugacy class of $G_p\subset\operatorname{GL}(T_{X,p})$.
\end{lemma}

\begin{proof}
By Cartan's linearization theorem, the germ $(X,p)$ is $G_p$-equivariantly
isomorphic to $(T_{X,p},0)$ with the linear action. Since $\rho_p^M$ is
defined by data over a neighborhood of $p$, an admissible model transports
along this isomorphism to an admissible model at $0$ with the same class. We
may therefore assume that the germ is linear, and realize it globally. Let
$G:=G_p$ act on $\mathbb P^2$ by $g[z_0:z_1:z_2]=[z_0:g(z_1,z_2)]$, using the
tangent representation, and take $p=[1:0:0]$. This action is effective. No
nonidentity element has eigenvalue $1$ on $T_{X,p}$, so $G$ acts freely off
$p$ on the invariant open set $U_0:=\mathbb P^2\setminus L_\infty$, where
$L_\infty=\{z_0=0\}$.

A point $q\in L_\infty$ corresponds to a line in $T_{X,p}$. Its isotropy
subgroup preserves this line and an invariant complement, so, as in the
proof of Lemma~\ref{lem:diag-formal}, it acts diagonally near $q$. The only
possible fixed curve is $L_\infty$ itself, because a fixed line through $p$
would require the eigenvalue $1$. By
Proposition~\ref{prop:general-hj-natural-residues}, the minimal resolution
of $\mathbb P^2/G$ near the image of $L_\infty$ therefore carries logarithmic
natural extensions, along an SNC divisor.

Let $M_1,M_2$ be admissible models at $p$, defined over a $G$-stable
neighborhood $U\subset U_0$ of $p$. For $i=1,2$, the blow-ups in $M_i$ are
blow-ups of points over $p$, so they extend to a $G$-equivariant composition
of point blow-ups $X'_i\to\mathbb P^2$ that is an isomorphism off $p$. Near
the image of $p$, the resolution in $M_i$ factors through the minimal
resolution of $X'_i/G$ followed by blow-ups of points, since every resolution
of a normal surface singularity factors through the minimal one and a proper
birational morphism of smooth surfaces is a composition of point blow-ups
\cite[Chapter~V, \S~5]{hartshorne_algebraic_1977}. Performing the same
point blow-ups on the minimal resolution of $X'_i/G$, which near the image of
$L_\infty$ is the fixed resolution above, gives a projective resolution
$\mu_i\colon S_i\to Y'_i:=X'_i/G$ that agrees with $M_i$ over $U/G$. The natural extensions
$\mathcal V_\xi^{(i)}:=\mu_i^*\mathcal F_\xi/\operatorname{torsion}$ are
logarithmic along an SNC divisor $\Delta_i$. By
Proposition~\ref{extension_prop} and equivariant birational invariance,
\[
    \chi\bigl(S_i,\mathcal V_\xi^{(i)}\bigr)
    =\chi(Y'_i,\mathcal F_\xi)
    =\bigl\langle\chi_G(\mathbb P^2,\mathcal O),\xi\bigr\rangle_G,
\]
and $\chi(\mathcal O_{S_i})=\chi(\mathcal O_{\mathbb P^2/G})$ because
quotient singularities are rational. Thus the right side of
\eqref{eq:general-hj-global-residue-correction} for $\mathcal V_\xi^{(i)}$
is the same for $i=1,2$.

Split it into the terms involving components of $\Delta_{\bar p}^{(i)}$ and
the rest. The first part is $\langle\rho_p^{M_i},\xi\rangle_G$. The
complements $S_i\setminus\Delta_{\bar p}^{(i)}$ are identified with each
other, together with the natural extensions and their connections. Every
other component of $\Delta_i$ is a proper curve in this common open set,
disjoint from $\Delta_{\bar p}^{(i)}$, so its self-intersection, canonical
degree, and residues are the same for $i=1,2$. Hence the remaining terms
coincide, and $\langle\rho_p^{M_1},\xi\rangle_G=\langle\rho_p^{M_2},\xi\rangle_G$
for all $\xi\in\operatorname{Irr}(G)$. Since $G=G_p$, this proves
$\rho_p^{M_1}=\rho_p^{M_2}$.

For the second assertion, equivariantly isomorphic germs have admissible
models with the same class, by transport of structure. By the first
assertion, $\rho_p$ therefore depends only on the isomorphism class of the
germ, which by Cartan's theorem is determined by the conjugacy class of
$G_p\subset\operatorname{GL}(T_{X,p})$.
\end{proof}

We write $\rho_p$ for this class. In particular, at a cyclic quotient
point the direct calculations of Proposition~\ref{prop:An-local-correction}
and Proposition~\ref{prop:general-hj-residue-formula} agree with the
blow-up calculation \eqref{eq:general-hj-global-assembly}, and the classes of
Theorem~\ref{thm:blowup-assembly} are intrinsic to the rational double
points. When fixed curves pass through $p$, the exceptional part does not need to
be independent of the model, since blowing up changes the self-intersections
of the branch curves. Only its sum with the branch terms is intrinsic.

\subsection{Checks and the resulting procedure}

The scaling and orientation conventions are illustrated by the
$\frac15(1,2)$ calculation of Section~\ref{subsec:applications-local}.
For a binary polyhedral arm with $G_q=C_k$, take $\lambda$ to be the
character on the normal coordinate function. Thus
$\alpha=\lambda$, $\beta=\lambda^{-2}$, $r_x=r_E=2$, $r_y=1$, and
the residual action has $n=k$ and $a=k-1$. For
$\chi=\lambda^{2r+d}$, with $0\leq r<k$ and $d\in\{0,1\}$, the character
data are simply $d_x=d$, $d_y=0$, and $c_\chi=r$.
The residue formula gives \eqref{eq:even-arm-residues} and
\eqref{eq:odd-arm-residues}, and hence the arm class $\alpha_k$ of
Proposition~\ref{prop:cyclic-arm-correction}.
Since $T^2=-2$ and $K_S\cdot T=0$ in these cases, the central term is
$\sigma_p=\frac14\operatorname{Ind}_{\{\pm I\}}^{G_p}\varepsilon$,
as in \eqref{eq:central-reflection-correction}. Thus
\eqref{eq:general-hj-global-assembly} recovers
Theorem~\ref{thm:blowup-assembly}.

In general, the calculation is finite. For each special orbit, determine
$r_x,r_y,n,a$ and the character data $d_x,d_y,c_\chi$. Compute the rays
from \eqref{eq:general-hj-ray-recurrence}, then the residues from
\eqref{eq:general-hj-finite-residue-formula}. Insert them into
\eqref{eq:general-hj-universal-correction}, compute the central term from
\eqref{eq:general-hj-central-square-resolved}, and assemble by
\eqref{eq:general-hj-global-assembly}. When fixed curves are present,
retain their remaining terms in the global formula. This gives an explicit
procedure without requiring a separate logarithmic modification in the
diagonal case.

\subsection{Proof of the global formula}
\label{subsec:general-hj-global-proof}

\begin{proof}[Proof of Theorem~\ref{thm:intro-global}]
First take a $G$-equivariant log resolution of the support of $D$, as in
Section~\ref{sec:reduction}. The generic isotropy subgroup along a fixed
curve is cyclic, since it acts faithfully on the normal line. Hence the
points with non-abelian isotropy subgroup are isolated, and there are
finitely many of them. Blow up their full $G$-orbits. By
Lemma~\ref{lem:diag-formal}, all the new isotropy subgroups act diagonally.
The remaining isotropy subgroups are abelian, hence diagonalizable after
linearization.

In a diagonal chart, fixed curves are coordinate axes, so they meet
transversely. If distinct fixed curves in the same $G$-orbit meet, blow up
the orbits of their intersection points. This separates their strict
transforms, and Lemma~\ref{lem:diag-formal} again shows that the new
isotropy subgroups act diagonally. The exceptional curves created in each
orbit are mutually disjoint and are not translates of an older fixed
curve. These blow-ups also preserve the SNC support of the pulled-back
divisor. This gives the required $f\colon X'\to X$.

The local quotients of $X'$ are the diagonal models treated above, and the
branch curves are images of coordinate axes. A fixed curve is smooth, and
after the separation of translates two points of it are identified in the
quotient only by its setwise stabilizer. Hence each branch curve on $Y'$ is
smooth. The toric resolutions have smooth exceptional curves and transverse
boundary intersections, so the full divisor $\Delta$ is SNC.

The multiplicity sheaf $\mathcal F_\xi$ is reflexive, as a direct summand
of the reflexive sheaf $\pi'_*\mathcal O_{X'}$. On a small analytic
neighborhood of the image of $q\in X'$, Frobenius reciprocity gives its
decomposition into the local character modules. Their natural extensions
therefore give
\[
    \mathcal V_\xi
    \cong\bigoplus_{\chi\in\operatorname{Irr}(G_q)}
       \bigl(\mathcal L^{\mathrm{nat}}_\chi\bigr)^{
       \oplus\langle\operatorname{Res}_{G_q}^G\xi,\chi\rangle_{G_q}}
\]
locally over $Y'$. The multiplicities sum to $\xi(1)$ because $G_q$ is
abelian. Proposition~\ref{prop:general-hj-natural-residues} gives a
logarithmic connection on each summand, induced by the same connection
$d$. Hence the global natural extension has rank $\xi(1)$ and is
logarithmic along $\Delta$. Proposition~\ref{extension_prop} and
\eqref{eq:general-hj-global-residue-correction} give
\[
    \chi(Y',\mathcal F_\xi)
      =\chi(S,\mathcal V_\xi)
      =\xi(1)\chi(S,\mathcal O_S)-m_\xi(\mathcal O_X).
\]
Since $Y$ and $Y'$ are birational surfaces with rational singularities,
$\chi(S,\mathcal O_S)=\chi(Y,\mathcal O_Y)$. The multiplicity-sheaf
decomposition and equivariant birational invariance now give
\[
    \chi_G(X,\mathcal O_X)
    =\chi(Y,\mathcal O_Y)\chi_{\mathrm{reg}}
       -\sum_{\xi\in\operatorname{Irr}(G)}m_\xi(\mathcal O_X)\xi.
\]
Finally, apply Proposition~\ref{prop:ordered-peeling} to $f^*D$ and use
$f^*\mathcal L\simeq\mathcal O_{X'}(f^*D)$ with its natural divisor
linearization. Equivariant birational invariance identifies its Euler
characteristic with $\chi_G(X,\mathcal L)$, giving
\eqref{eq:intro-global-formula}.
\end{proof}

%% file: ch-applications.tex
\ProvidesFile{ch-applications.tex}

\section{Examples and Applications}
\label{sec:applications}

We finish with some applications of the preceding calculations. First, we compare the local classes with the holomorphic Lefschetz formula and recover two of the formulas for Kleinian surface singularities in \cite{lim_riemannroch_2023}. We then work out the local calculation for $\frac15(1,2)$ and use it in a global example on $\mathbb P^2$. Finally, we treat an action of $S_3$ on $\mathbb P^2$, where the pre-quotient blow-up is needed and fixed curves pass through the blown-up point. Throughout, character labels refer to the action on coordinate functions, as in Section~\ref{sec:general-hj}.

\subsection{Comparison with the Lefschetz and Kleinian coefficients}

\paragraph{The Lefschetz class function.}
Let $p$ be a point with isotropy group $G_p$, and let $G_p$ act on the cotangent space $T_{X,p}^*$ through its action on coordinate functions. If the coordinate functions have characters $\alpha,\beta$, then $\det(1-g\,|\,T_{X,p}^*)=(1-\alpha(g))(1-\beta(g))$. When $p$ is an isolated fixed point of every nonidentity element of $G_p$, this determinant is nonzero for $g\neq1$, and we define a class function on $G_p$ by
\begin{equation}
\label{eq:applications-lefschetz-class}
    \rho_p^{L}(g)
    =-\frac1{\det(1-g\,|\,T_{X,p}^*)}\quad(g\neq1),
    \qquad
    \rho_p^{L}(1)
    =\sum_{h\neq1}\frac1{\det(1-h\,|\,T_{X,p}^*)}.
\end{equation}
Suppose that every point of $X$ with nontrivial isotropy is isolated. Then
\begin{equation}
\label{eq:applications-lefschetz-global}
    \sum_{\xi\in\operatorname{Irr}(G)}m_\xi(\mathcal O_X)\,\xi
    =\sum_{[p]}\operatorname{Ind}_{G_p}^G\rho_p^{L},
\end{equation}
where the sum runs over the $G$-orbits of such points. Indeed, by \eqref{eq:intro-global-formula} with $\mathcal L=\mathcal O_X$, the left side is the class function $\chi(\mathcal O_Y)\chi_{\mathrm{reg}}-\chi_G(X,\mathcal O_X)$. For $g\neq1$, the holomorphic Lefschetz formula \cite{atiyah_lefschetz_1968,atiyah_index_1968} gives
\[
    \chi_G(X,\mathcal O_X)(g)
    =\sum_{p\in X^g}\frac1{\det(1-g\,|\,T_{X,p}^*)},
\]
which is the negative of the right side at $g$. At $g=1$, the right side is
\[
    \sum_{g\neq1}\chi_G(X,\mathcal O_X)(g)
    =|G|\chi(\mathcal O_Y)-\chi(\mathcal O_X),
\]
because $\langle\chi_G(X,\mathcal O_X),\mathbf1\rangle_G=\chi(\mathcal O_Y)$.

By Lemma~\ref{lem:general-hj-model-independence}, $\rho_p$ depends only on the local action at $p$, not on the model used to compute it. Thus the $D_4$ class below, which comes from the pre-quotient blow-up, and the cyclic classes, which are computed directly on the quotient, are directly comparable. The class $\rho_p^{L}$ is also local, and \eqref{eq:applications-lefschetz-global} identifies the sums of the two over all orbits. We do not prove the local equality $\rho_p=\rho_p^{L}$ in general. In each example below with isolated fixed points, we check it directly.

\paragraph{The Kleinian coefficients.}
Let $G_p\subset\operatorname{SL}_2(\mathbb C)$, and let $W$ be its natural two-dimensional representation. For an irreducible character $\eta$, the coefficients of Lim and Rota can be written as
\[
    T_\eta
    =\frac1{|G_p|}\sum_{g\neq1}
      \frac{\eta(g)}{2-\operatorname{tr}(g|W)}.
\]
Lim and Rota obtain them from To\"en's Riemann--Roch theorem for Deligne--Mumford stacks \cite[Section~2]{lim_riemannroch_2023}. For a global quotient, these are the local terms of Kawasaki's orbifold Riemann--Roch formula \cite{kawasaki_riemann-roch_1979}. Their normalization differs from ours. Since $W\cong T_{X,p}^*$ and $\det(1-g|W)=2-\operatorname{tr}(g|W)$ for $g\in\operatorname{SL}_2(\mathbb C)$, expanding the inner product with \eqref{eq:applications-lefschetz-class} gives
\[
    \langle\rho_p^{L},\eta\rangle_{G_p}
    =\eta(1)T_{\mathbf1}-T_{\bar\eta}.
\]
Moreover, $T_{\bar\eta}=T_\eta$: replace $g$ by $g^{-1}$ and use $\operatorname{tr}(g^{-1}|W)=\operatorname{tr}(g|W)$. Hence the local equality $\rho_p=\rho_p^{L}$ is equivalent to
\begin{equation}
\label{eq:applications-normalization}
    m_{p,\eta}=\eta(1)T_{\mathbf1}-T_\eta
    \qquad(\eta\in\operatorname{Irr}(G_p)).
\end{equation}
This is the expected relation: our correction is measured relative to $\eta(1)\chi(\mathcal O_Y)$, while $T_{\mathbf1}$ is the contribution to $\chi(\mathcal O_Y)$ itself.

The regular character vanishes at $g\neq1$, so
\[
    \sum_\eta\eta(1)T_\eta=0,
    \qquad
    T_{\mathbf1}
    =\frac1{|G_p|}\sum_\eta\eta(1)m_{p,\eta}
\]
whenever \eqref{eq:applications-normalization} holds. Thus our coefficients also determine the constant $T_{\mathbf1}$.

\paragraph{The cyclic case.}
For a singularity of type $A_{N-1}$, choose a generator $g$ and a character $\chi$ so that the coordinate functions have characters $\chi,\chi^{-1}$, with $\chi(g)=\exp(2\pi i/N)$. Proposition~\ref{prop:An-local-correction} gives
\[
    m_{p,\chi^j}=\frac{j(N-j)}{2N},
    \qquad 0\leq j<N.
\]
Averaging these coefficients gives $T_{\mathbf1}=(N^2-1)/(12N)$, and \eqref{eq:applications-normalization} predicts
\begin{equation}
\label{eq:applications-cyclic-T}
    T_{\chi^j}
    =\frac{N^2-1}{12N}-\frac{j(N-j)}{2N}.
\end{equation}
This is the formula in \cite[Section~4.1]{lim_riemannroch_2023}, where the cyclic calculation is attributed to Lieblich. Hence $\rho_p=\rho_p^{L}$ for every $A_{N-1}$ singularity.

\paragraph{The $D_4$ case.}
Here $G_p=Q_8$, and there are three special isotropy orbits, each with a cyclic isotropy group of order four. At one orbit representative, choose a generator $g$ and write $\lambda(g)=i$ for the character on the normal coordinate function. Thus
\[
    g\cdot x=ix,\qquad g\cdot y=-y.
\]
Proposition~\ref{prop:cyclic-arm-correction} gives the arm class
\[
    \alpha_2=-\frac1{16}\lambda
              +\frac14\lambda^2
              +\frac3{16}\lambda^3.
\]
The central curve contributes $\frac14\operatorname{Ind}_{\{\pm I\}}^{Q_8}\varepsilon$, where $\varepsilon$ is the nontrivial character of $\{\pm I\}$.

Write $\varepsilon_1,\varepsilon_2,\varepsilon_3$ for the nontrivial characters of degree one, and $\theta$ for the irreducible character of degree two, as in Section~\ref{sec:obstruction}. Each $\varepsilon_i$ restricts to $\lambda^2$ on two of the three isotropy subgroups and trivially on the third. Also, $\theta$ restricts to $\lambda+\lambda^3$ on each subgroup and to $2\varepsilon$ on $\{\pm I\}$. The assembly formula and Frobenius reciprocity therefore give
\[
\begin{aligned}
    m_{p,\mathbf1}&=0,\qquad
    m_{p,\varepsilon_i}=2\left(\frac14\right)=\frac12,\\
    m_{p,\theta}&=\frac12+3\left(-\frac1{16}+\frac3{16}\right)=\frac78.
\end{aligned}
\]
It follows that
\[
    T_{\mathbf1}
    =\frac18\left(3\cdot\frac12+2\cdot\frac78\right)
    =\frac{13}{32}.
\]
Using \eqref{eq:applications-normalization}, we obtain
\[
    T_{\varepsilon_i}=-\frac3{32},
    \qquad
    T_\theta=-\frac1{16},
\]
as in \cite[Section~4.2.1]{lim_riemannroch_2023}. These values can also be read off directly: $2-\operatorname{tr}(g|W)$ equals $4$ at $-I$ and $2$ at the six elements of order four. Hence $\rho_p=\rho_p^{L}$ for $D_4$. The same comparison can be carried out for the other binary polyhedral groups, using Theorem~\ref{thm:blowup-assembly}.

\subsection{The local calculation for \texorpdfstring{$\frac15(1,2)$}{1/5(1,2)}}
\label{subsec:applications-local}

Let $p$ be a fixed point with isotropy group $\langle g\rangle\cong\mathbb Z_5$, and choose coordinates with
\[
    g\cdot x=\zeta_5x,\qquad g\cdot y=\zeta_5^2y,
    \qquad \xi(g)=\zeta_5.
\]
We apply the diagonal calculation of Section~\ref{sec:general-hj} directly to this quotient. No preliminary blow-up is needed. There are no quasi-reflections, so $r_x=r_y=1$ and the coordinates $X,Y$ of that section are just $x,y$. Also, $n=5$, $a=2$, and $a^\vee=3$. The continued fraction $5/3=[2,3]$, ordered from $x=0$ to $y=0$, gives
\[
    v_1=\left(\frac35,\frac15\right),\qquad
    v_2=\left(\frac15,\frac25\right),\qquad
    J=\begin{pmatrix}-2&1\\1&-3\end{pmatrix},\qquad
    \boldsymbol\kappa=\begin{pmatrix}0\\1\end{pmatrix}.
\]
Here $J$ is the intersection matrix of the exceptional chain, and $\boldsymbol\kappa$ records its canonical intersections.

For $\chi=\xi^c$, with $0\leq c<5$, the character data of Section~\ref{sec:general-hj} are $d_x=d_y=0$ and $c_\chi=c$. The monomials in the character module satisfy
\[
    x^u y^v:\qquad u+2v\equiv c\pmod5.
\]
Taking their minimum valuations along $v_1,v_2$ gives the residues below. Since both boundary residues vanish, Proposition~\ref{prop:general-hj-residue-formula} gives
\[
    m_{p,\chi}
    =-\frac12\left(
       \mathbf R_\chi^T J\mathbf R_\chi
       +\boldsymbol\kappa^T\mathbf R_\chi
      \right),
    \qquad
    \mathbf R_\chi=(R_{1,\chi},R_{2,\chi})^T.
\]
Thus
\[
\renewcommand{\arraystretch}{1.25}
\begin{array}{c|cc|c}
 \chi&R_{1,\chi}&R_{2,\chi}&m_{p,\chi}\\ \hline
 \mathbf1&0&0&0\\
 \xi&3/5&1/5&1/5\\
 \xi^2&1/5&2/5&0\\
 \xi^3&4/5&3/5&2/5\\
 \xi^4&2/5&4/5&2/5
\end{array}
\]
For example, for $\chi=\xi^3$, the monomial $xy$ has valuations $(4/5,3/5)$ and attains both minima. Therefore
\[
    J\mathbf R_{\xi^3}=\begin{pmatrix}-1\\-1\end{pmatrix},
    \qquad
    m_{p,\xi^3}
    =-\frac12\left(-\frac75+\frac35\right)
    =\frac25.
\]
The full local correction is
\begin{equation}
\label{eq:applications-local-5-2}
    \rho_p=\frac15\xi+\frac25\xi^3+\frac25\xi^4.
\end{equation}
In particular, a nontrivial character can have zero correction, as happens for $\xi^2$. One checks that
\[
    \rho_p(g^k)=-\frac1{(1-\zeta_5^k)(1-\zeta_5^{2k})},
    \qquad 1\leq k\leq4,
\]
and $\rho_p(1)=1$, so again $\rho_p=\rho_p^{L}$.

\subsection{A global example on \texorpdfstring{$\mathbb P^2$}{P2}}
\label{subsec:applications-global}

Let $X=\mathbb P^2$, and let $G=\langle g\rangle\cong\mathbb Z_5$ act on points by
\[
    g[z_0:z_1:z_2]
      =[z_0:\zeta_5^{-1}z_1:\zeta_5^{-2}z_2].
\]
Put $\xi(g)=\zeta_5$. The negative exponents ensure that the coordinate functions $z_1/z_0,z_2/z_0$ have characters $\xi,\xi^2$ under the action $g\cdot h=h\circ g^{-1}$.

Every nonidentity element has three distinct eigenvalues, so the only points with nontrivial isotropy are
\[
    p_0=[1:0:0],\qquad p_1=[0:1:0],\qquad p_2=[0:0:1].
\]
Each is fixed by all of $G$. Thus $Y=X/G$ has three isolated quotient singularities and no divisorial branch locus. We choose local coordinates as follows:
\[
\renewcommand{\arraystretch}{1.25}
\begin{array}{c|cc|cc}
 &x&y&\text{character of }x&\text{character of }y\\ \hline
 p_0&z_1/z_0&z_2/z_0&\xi&\xi^2\\
 p_1&z_2/z_1&z_0/z_1&\xi&\xi^{-1}\\
 p_2&z_1/z_2&z_0/z_2&\xi^{-1}&\xi^{-2}
\end{array}
\]

\paragraph{The three local corrections.}
At $p_0$, equation~\eqref{eq:applications-local-5-2} gives
\[
    \rho_{p_0}=\frac15\xi+\frac25\xi^3+\frac25\xi^4.
\]
At $p_1$, the quotient has type $A_4$. Hence
\[
    \rho_{p_1}
    =\sum_{j=0}^4\frac{j(5-j)}{10}\xi^j
    =\frac25\xi+\frac35\xi^2+\frac35\xi^3+\frac25\xi^4.
\]
At $p_2$, the same $\frac15(1,2)$ calculation applies with $\xi$ replaced by $\xi^{-1}$. Equivalently, use the generator $g^{-1}$. Expressing the result in the original characters of $G$ gives
\[
\begin{aligned}
    \rho_{p_2}
    &=\frac15\xi^{-1}+\frac25\xi^{-3}+\frac25\xi^{-4}\\
    &=\frac25\xi+\frac25\xi^2+\frac15\xi^4.
\end{aligned}
\]
There is no induction to perform because each isotropy group is $G$. Adding the three classes, we obtain
\begin{equation}
\label{eq:applications-global-correction}
    \rho_{p_0}+\rho_{p_1}+\rho_{p_2}
      =\xi+\xi^2+\xi^3+\xi^4
      =\chi_{\mathrm{reg}}-\mathbf1.
\end{equation}
Since each $\rho_{p_i}$ equals $\rho_{p_i}^{L}$, evaluating \eqref{eq:applications-global-correction} at $g^k\neq1$ gives the Lefschetz identity $\sum_i\det(1-g^k|T_{X,p_i}^*)^{-1}=1$.

\paragraph{The structure sheaf.}
Let $S\to Y$ be the minimal resolution. The quotient is rational: on the chart at $p_0$,
\[
    \mathbb C(Y)=\mathbb C(x,y)^G
      =\mathbb C\left(x^5,\frac{y}{x^2}\right).
\]
Indeed, adjoining $x$ recovers $\mathbb C(x,y)$ and gives a degree-five extension. Thus $S$ is a smooth rational projective surface and $\chi(\mathcal O_S)=1$. Since quotient singularities are rational, $\chi(\mathcal O_Y)=1$ as well.

The natural isotypic extensions on $S$ are logarithmic by Proposition~\ref{prop:general-hj-natural-residues}. There are no branch-curve terms, so the global residue formula gives
\[
\begin{aligned}
    \chi_G(\mathbb P^2,\mathcal O_{\mathbb P^2})
    &=\chi(\mathcal O_Y)\chi_{\mathrm{reg}}
       -\rho_{p_0}-\rho_{p_1}-\rho_{p_2}\\
    &=\mathbf1.
\end{aligned}
\]
This agrees with the direct calculation: $H^0(\mathbb P^2,\mathcal O)=\mathbb C$ with the trivial action, and the higher cohomology vanishes. The example also shows how fractional local coefficients combine to give an integral global character.

\paragraph{Divisor peeling for $\mathcal O(d)$.}
Let $D=\{z_0=0\}$, and give $\mathcal O_X(dD)\cong\mathcal O_{\mathbb P^2}(d)$ its natural divisor linearization. The linear action above induces characters $\mathbf1,\xi,\xi^2$ on the sections $z_0,z_1,z_2$ of $\mathcal O(1)$. Since $z_0$ is invariant and has divisor $D$, this agrees with the chosen divisor linearization.

For $d\geq0$, the peeling formula gives
\[
    \chi_G(X,\mathcal O_X(dD))
      =\mathbf1+\sum_{k=1}^d
        \chi_G(D,\mathcal O_D(kD)).
\]
Here $D\cong\mathbb P^1$ and $\mathcal O_D(kD)\cong\mathcal O_{\mathbb P^1}(k)$. Its higher cohomology vanishes, and its sections have basis
\[
    z_1^b z_2^{k-b},\qquad 0\leq b\leq k,
\]
with characters $\xi^{b+2(k-b)}$. Consequently,
\begin{equation}
\label{eq:applications-projective-line-bundles}
    \chi_G(\mathbb P^2,\mathcal O(d))
      =\mathbf1+\sum_{k=1}^d\sum_{b=0}^k\xi^{b+2(k-b)}
      =\sum_{\substack{b,c\geq0\\b+c\leq d}}\xi^{b+2c}.
\end{equation}
For instance,
\[
    \chi_G(\mathbb P^2,\mathcal O(1))=\mathbf1+\xi+\xi^2.
\]
More generally, \eqref{eq:applications-projective-line-bundles} is exactly the character of the homogeneous monomial basis $z_0^{d-b-c}z_1^b z_2^c$. Evaluating at the identity gives $\binom{d+2}{2}$, the ordinary Euler characteristic. Thus the local residue calculation supplies the structure-sheaf term, and divisor peeling supplies the remaining line-bundle contributions.

\subsection{A non-abelian example with fixed curves}
\label{subsec:applications-S3}

The previous examples have only isolated fixed points. We now give an example in which the pre-quotient blow-up is required and fixed curves pass through the blown-up point, so that the branch-curve terms of Section~\ref{sec:general-hj} also enter.

Let $G=S_3$ act on $X=\mathbb P^2$ by permuting the homogeneous coordinates, and write $\mathbf1,\operatorname{sgn},\operatorname{st}$ for its irreducible characters, of degrees $1,1,2$. The elementary symmetric polynomials identify $Y=\mathbb P^2/S_3$ with the weighted projective plane $\mathbb P(1,2,3)$, so $\chi(\mathcal O_Y)=1$. Since $\chi_G(\mathbb P^2,\mathcal O)=\mathbf1$, the procedure must give
\begin{equation}
\label{eq:applications-S3-target}
    \sum_\xi m_\xi(\mathcal O_X)\,\xi
    =\chi_{\mathrm{reg}}-\mathbf1
    =\operatorname{sgn}+2\operatorname{st}.
\end{equation}

\paragraph{Points with nontrivial isotropy.}
Put $\omega=\zeta_3$. A transposition $(ij)$ fixes the line $\ell_{ij}=\{z_i=z_j\}$ pointwise, together with one further point. For instance, $(01)$ fixes $[1:-1:0]$. A $3$-cycle fixes $[1:1:1]$, $[1:\omega:\omega^2]$, and $[1:\omega^2:\omega]$. Thus the points with nontrivial isotropy are:
\begin{itemize}
    \item the point $p=[1:1:1]$, with $G_p=S_3$ acting on $T_{X,p}$ by its reflection representation
    \item the orbit of $[1:-1:0]$, of size three, with isotropy $\langle(01)\rangle$ acting by $-1$, so of type $A_1$
    \item the orbit of $[1:\omega:\omega^2]$, of size two, with isotropy $A_3$, of type $A_2$
    \item the points of the three lines $\ell_{ij}$, with generic isotropy of order two.
\end{itemize}
The three lines form one $G$-orbit and meet pairwise only at $p$. The image of $p$ in $Y$ is a smooth point by the Chevalley--Shephard--Todd theorem, but the branch curve has a cusp there: it is the discriminant curve of the reflection representation. Since $G_p$ is non-cyclic and the fixed curves in one orbit meet at $p$, we blow up $p$.

\paragraph{The exceptional part at $p$.}
Let $f\colon X'\to X$ be the blow-up at $p$. The strict transforms $\ell'_{ij}$ are pairwise disjoint and satisfy $(\ell'_{ij})^2=0$. The reflection representation of $S_3$ contains no nontrivial scalars, so $H_E$ is trivial. Hence $E$ is not a fixed curve, the central term $\sigma_p$ vanishes, and the strict transform $T$ of the image of $E$ is not a component of $\Delta$.

The group acts on $E=\mathbb P(T_{X,p})$ as the dihedral group of order six. Its three special orbits consist of the two eigenlines of a $3$-cycle $c$, the three mirror lines of the transpositions, and the three $(-1)$-eigenlines of the transpositions. We use the coordinates $(u,v)$ of Lemma~\ref{lem:diag-formal}. Let $t$ be a transposition, $\varepsilon_t$ the nontrivial character of $\langle t\rangle$, and $\psi$ the character of $\langle c\rangle$ with $\psi(c)=\omega$.
\begin{itemize}
    \item At a mirror line, $\lambda(t)=1$ and $\mu(t)=-1$, so $t\cdot(u,v)=(u,-v)$. This is a reflection whose fixed curve $\{v=0\}$ is the strict transform of $\ell_{ij}$. Thus $r_x=1$, $r_y=2$, and $n=1$. The quotient point is smooth, and its contribution $-t_{x,\chi}t_{y,\chi}$ vanishes because $t_{x,\chi}=0$.
    \item At a $(-1)$-eigenline, $\lambda(t)=-1$ and $\mu(t)=1$, so $t\cdot(u,v)=(-u,-v)$. The quotient has type $A_1$, and Proposition~\ref{prop:An-local-correction} gives $\rho_q^{\mathrm{loc}}=\frac14\varepsilon_t$.
    \item At the eigenline on which $c$ has tangent eigenvalue $\omega$, we have $\lambda(c)=\omega^2$ and $\mu(c)=\omega$, so $c\cdot(u,v)=(\omega^2u,\omega^2v)$. The quotient has type $\frac13(1,1)$, with generator $c^2$, $n=3$, and $a=a^\vee=1$. The chain is one curve with $J=(-3)$ and $\boldsymbol\kappa=(1)$, and $v_1=(\frac13,\frac13)$. For $\chi(c^2)=\omega^e$, the residue is $e/3$, and Proposition~\ref{prop:general-hj-residue-formula} gives $m_{q,\chi}=e(e-1)/6$. Thus $\rho_q^{\mathrm{loc}}=\frac13\psi$.
\end{itemize}
The characters restrict as follows: $\operatorname{sgn}$ restricts to $\varepsilon_t$ and to $\mathbf1$, and $\operatorname{st}$ restricts to $\mathbf1+\varepsilon_t$ and to $\psi+\psi^2$. Formula \eqref{eq:general-hj-global-assembly} and Frobenius reciprocity therefore give the exceptional part
\[
    \langle\rho_p^{\mathrm{exc}},\mathbf1\rangle=0,
    \qquad
    \langle\rho_p^{\mathrm{exc}},\operatorname{sgn}\rangle=\frac14,
    \qquad
    \langle\rho_p^{\mathrm{exc}},\operatorname{st}\rangle=\frac14+\frac13=\frac7{12}.
\]
Although $p$ has smooth image in $Y$, its exceptional part is nonzero.

\paragraph{The remaining terms.}
By Proposition~\ref{prop:An-local-correction}, the $A_1$ orbit contributes $\frac14\varepsilon_t$ on $\langle t\rangle$, hence $\frac14$ to both $\operatorname{sgn}$ and $\operatorname{st}$. The $A_2$ orbit contributes $\frac13(\psi+\psi^2)$ on $A_3$, hence $0$ to $\operatorname{sgn}$ and $\frac23$ to $\operatorname{st}$.

Let $B$ be the image of the three lines $\ell'_{ij}$ in $Y'=X'/G$. Since $\pi'^*B=2\sum\ell'_{ij}$ and the $\ell'_{ij}$ are disjoint curves of self-intersection zero, the projection formula gives $6B^2=0$. The strict transforms $\ell'_{ij}$ meet $E$ only at the mirror orbit, where the quotient is smooth, and every other point of $\ell'_{ij}$ has isotropy exactly the reflection group $\langle t\rangle$. Hence $B$ lies in the smooth locus of $Y'$. Its strict transform on $S$ therefore satisfies
\[
    B^2=0,
    \qquad
    K_S\cdot B=-2.
\]
Along $B$, the residue is $\frac12$ on the $\varepsilon_t$-part of the multiplicity bundle and $0$ on the rest. The $\varepsilon_t$-part of $\mathcal V_\xi$ has rank $\langle\operatorname{Res}_{\langle t\rangle}^G\xi,\varepsilon_t\rangle$, which is $0,1,1$ for $\mathbf1,\operatorname{sgn},\operatorname{st}$. Hence the branch term $-\frac12\bigl(\operatorname{tr}(R^2)B^2+\operatorname{tr}(R)K_S\cdot B\bigr)$ contributes $\frac12$ to both $\operatorname{sgn}$ and $\operatorname{st}$. The curve $B$ meets no other component of $\Delta$. It meets $T$, which is not in $\Delta$.

Adding the four contributions gives
\[
\begin{aligned}
    m_{\operatorname{sgn}}(\mathcal O_X)
    &=\frac14+\frac14+0+\frac12=1,\\
    m_{\operatorname{st}}(\mathcal O_X)
    &=\frac7{12}+\frac14+\frac23+\frac12=2,
\end{aligned}
\]
and $m_{\mathbf1}(\mathcal O_X)=0$, in agreement with \eqref{eq:applications-S3-target}. Unlike the $\mathbb Z_5$ example, a global branch-curve term appears here. It uses $B^2$ and $K_S\cdot B$ on $S$ and contributes $\frac12$ to each coefficient, that is, half of $m_{\operatorname{sgn}}$ and a quarter of $m_{\operatorname{st}}$. This is the allocation of Section~\ref{sec:general-hj}: the exceptional part is computed locally, while the self-intersection and canonical terms of the branch curves remain in the global formula.